\documentclass[12pt]{amsart}
\usepackage{amsfonts}
\usepackage{graphicx}
\usepackage{tabularx}
\usepackage{array}
\usepackage[usenames,dvipsnames]{color}
\usepackage{comment}
\usepackage{amsmath}
\usepackage{amsthm}
\usepackage{amssymb}
\usepackage{fullpage}
\usepackage[dvipsnames]{xcolor}
\usepackage{tikz}
\usepackage{listings}

\DeclareMathOperator{\dist}{dist}
\DeclareMathOperator{\mad}{mad}

\newtheorem{theorem}{Theorem}[section]

\newtheorem{assumption}[theorem]{Assumption}

\newtheorem{question}[theorem]{Question}
\newtheorem{claim}[theorem]{Claim}
\newtheorem{lemma}[theorem]{Lemma}

\theoremstyle{definition}
\newtheorem{definition}[theorem]{Definition}

\def\epsilon{\varepsilon}

\title{Flexible list coloring of graphs with maximum average degree less than $3$}
\author{Richard Bi}
\address{Department of Mathematics, University of Illinois Urbana-Champaign}
\email{rbi3@illinois.edu}

\author{Peter Bradshaw}
\address{Department of Mathematics, University of Illinois Urbana-Champaign}
\email{pb38@illinois.edu}
\thanks{Peter Bradshaw received support from NSF RTG grant DMS-1937241}

\begin{document}
\maketitle
\begin{abstract}
In the \emph{flexible list coloring} problem, we consider a graph $G$ and a color list assignment $L$ on $G$, as well as a subset $U \subseteq V(G)$ for which each $u \in U$ has a preferred color $p(u) \in L(u)$.
Our goal is to find a proper $L$-coloring $\phi$ of $G$ such that $\phi(u) = p(u)$ for at least $\epsilon|U|$ vertices $u \in U$.
We say that $G$ is $\epsilon$-flexibly $k$-choosable
if for every $k$-size list assignment $L$ on $G$ and every subset of vertices with coloring preferences, $G$ has a proper $L$-coloring that satisfies an $\epsilon$ proportion of these coloring preferences.
Dvo\v{r}\'ak, Norin, and Postle [Journal of Graph Theory, 2019] asked whether every $d$-degenerate graph is $\epsilon$-flexibly $(d+1)$-choosable for some constant $\epsilon = \epsilon(d) > 0$.

In this paper, we prove that there exists a constant $\epsilon > 0$ such that every graph with maximum average degree less than $3$ is $\epsilon$-flexibly $3$-choosable, which gives a large class of $2$-degenerate graphs
which are $\epsilon$-flexibly $(d+1)$-choosable. In particular, our results imply a theorem of 
Dvo\v{r}\'ak, Masa\v{r}\'ik, Mus\'ilek, and Pangr\'ac [Journal of Graph Theory, 2020] stating that every planar graph of girth $6$ is $\epsilon$-flexibly $3$-choosable for some constant $\epsilon > 0$.
To prove our result, we generalize the existing reducible subgraph framework traditionally used for flexible list coloring
to allow reducible subgraphs of arbitrarily large order.
\end{abstract}
\section{Introduction}
\subsection{Background: Flexible list coloring}
Given a graph $G$, a \emph{proper coloring} of $G$ is an assignment of a color $\phi(v)$ to each vertex $v \in V(G)$ so that no two adjacent vertices receive the same color.
A graph $G$ is \emph{$k$-colorable} if $G$ has a proper coloring $\phi:V(G) \rightarrow \{1,\dots,k\}$, and such a function $\phi$ is called a \emph{$k$-coloring}.
Dvo\v{r}ak, Norin, and Postle \cite{DNP} observed that by permuting colors,
if $G$ is a $k$-colorable graph and some of the vertices in $G$ have a preferred color in $\{1,\dots,k\}$, then $G$ always has a $k$-coloring
that satisfies a positive proportion (namely $\frac{1}{k}$) of all color preferences. 

Given a graph $G$ and a color list $L(v) \subseteq \mathbb N$ for each vertex $v \in V(G)$,
we say that $G$ is \emph{$L$-colorable} if there exists a proper coloring $\phi:V(G) \rightarrow \mathbb N$ of $G$ such that $\phi(v) \in L(v)$ for each vertex $v \in V(G)$, and we call such a function $\phi$ an \emph{$L$-coloring}.
The \emph{list coloring problem} 
asks whether a given graph $G$ with a list assignment $L:V(G) \rightarrow 2^{\mathbb N}$ has a proper $L$-coloring. Dvo\v{r}\'ak, Norin, and Postle \cite{DNP} asked the following question in the setting of the list coloring problem.
Consider a graph $G$ and a 
list assignment $L:V(G) \rightarrow 2^{\mathbb N}$ for which $G$ is $L$-colorable. Suppose that
for some subset $U \subseteq V(G)$, each $u \in U$ has a preferred color $p(u) \in L(u)$.
Does there exist an $L$-coloring of $G$ that satisfies a given proportion of these coloring preferences?
This question turns out to be highly nontrivial and has led to many interesting questions and results.

If $G$ is a graph and $f:V(G) \rightarrow \mathbb N$ is a function,
 a mapping
$L:V(G) \rightarrow 2^{\mathbb N}$ is an \emph{$f$-assignment} on $G$ if $|L(v)| = f(v)$ for each $v \in V(G)$.
When $f$ is not specified, we say that $L$ is a \emph{list assignment}.
If $f(v) = k$ for all vertices $v \in V(G)$, then  $L$ is a \emph{$k$-assignment} on $G$.
We say that $G$ is \emph{$k$-choosable} if $G$ has an $L$-coloring for every $k$-assignment $L$ on $G$. The \emph{choosability} of $G$ is the least integer $k$ for which $G$ is $k$-choosable.

A \emph{weighted request}
on a graph $G$
with a list assignment $L$
is a function $w$ such that for each vertex $v \in V(G)$ and color $c \in L(v)$, $w$ maps the pair $(v,c)$ to a nonnegative real number $w(v,c)$.
Given a value $\epsilon > 0$, $G$ is \emph{weighted $\epsilon$-flexibly
$k$-choosable} if for every $k$-assignment $L$ and weighted request $w$ on $G$, there exists an $L$-coloring $\phi$ of $G$ such that 
\begin{eqnarray}
\label{eqn:epsilon}
\sum_{v \in V(G)} w(v,\phi(v)) \geq \epsilon \sum_{v \in V(G)} \sum_{c \in L(v)} w(v,c).
\end{eqnarray}
In other words, the weight of the pairs $(v,c)$ for which $\phi(v) = c$ is at least an $\epsilon$ proportion of the weight of all pairs $(v,c)$.

We say that $w$ is an \emph{(unweighted) request} if for each $v \in V(G)$, $w(v,c) = 1$ for at most one color $c \in L(v)$ and $w(v,c') = 0$ for all other colors $c' \in L(v)$.
Given a graph $G$, if there exists a value $\epsilon > 0$ such that the inequality (\ref{eqn:epsilon}) holds for every $k$-assignment $L$ on $G$ and every unweighted request $w$ on $G$, then $G$ is \emph{$\epsilon$-flexibly $k$-choosable}.

With this definition of flexible choosability established, 
the following meta-question arises naturally:
\begin{question}[\cite{DNP}]
\label{q:meta}
    Let $\mathcal G$ be a graph class for which every graph $G \in \mathcal G$ is $k$-choosable. Does there exist a value $\epsilon > 0$ for which every graph $G \in \mathcal G$ is weighted $\epsilon$-flexibly $k$-choosable?
\end{question}
Question \ref{q:meta} has led to a great volume of research and many nontrivial results.
For example, a result proven by Vizing \cite{Vizing} and independently by Erd\H{o}s, Rubin, and Taylor \cite{ERT} states that if $G$ is a connected graph of maximum degree $\Delta$ which is not a clique, then $G$ is $\Delta$-choosable. The second author, along with Masa\v{r}\'ik and Stacho \cite{BMS}, showed that this result also holds in the flexible list coloring setting, proving that such a graph $G$ is weighted $\frac{1}{2 \Delta^4}$-flexibly $\Delta$-choosable and $\frac{1}{6\Delta}$-flexibly $\Delta$-choosable.
The same authors also showed that every graph of treewidth $2$ is not only $3$-choosable, but also $\frac{1}{3}$-flexibly $3$-choosable.
%This question has been studied a lot. Several classes of planar graphs, Brooks theorem,
%$t$-trees,
%outerplanar, so on and so forth.

One natural graph class $\mathcal G$ to which Question \ref{q:meta} can be applied is the class of $d$-degenerate graphs. 
Given a graph $G$, we say that $G$ is \emph{$d$-degenerate} if every nonempty 
subgraph of $G$
has a vertex of degree at most $d$. It is well known that a $d$-degenerate graph $G$ has a linear vertex ordering such that each vertex $v \in V(G)$ has at most $d$ previous neighbors in the ordering. Therefore, a greedy coloring algorithm shows that every $d$-degenerate graph is $(d+1)$-choosable.
This fact
 leads to the following natural question:
    For each value $d \geq 1$, does there exist a value $\epsilon = \epsilon(d) > 0$ such that every $d$-degenerate graph is weighted $\epsilon$-flexibly $(d+1)$-choosable?

Unfortunately, this question seems to be out of reach using current methods.
While $1$-degenerate graphs are weighted $\frac{1}{2}$-flexibly $2$-choosable \cite{DNP},
it is unknown whether there exists a constant $\epsilon > 0$ for which $2$-degenerate graphs are $\epsilon$-flexibly $3$-choosable.
Even proving that a single preference can be satisfied
on a
$2$-degenerate graph
$G$ with a $3$-assignment $L$ is highly nontrivial,
and the only current proof \cite{DNP} of this fact relies on the Combinatorial Nullstellensatz of Alon and Tarsi \cite{AT}.
Due to the difficulty of working with the entire class of $d$-degenerate graphs, research often focuses on the following more specific question:
\begin{question}
\label{q:degen}
    Let $d \geq 1$ be an integer, and $\mathcal G$ be a fixed class of $d$-degenerate graphs. Does there exist a value $\epsilon = \epsilon(d) > 0$ such that every graph $G \in \mathcal G$ is 
    $\epsilon$-flexibly $(d+1)$-choosable?
\end{question}
Question \ref{q:degen}
has an affirmative answer for non-regular connected 
graphs of maximum degree $\Delta$, as these graphs
are $\Delta$-degenerate and
$\frac{1}{6 \Delta}$-flexibly $\Delta$-choosable \cite{BMS}.
The question also has an affirmative answer for triangle-free planar graphs, which are $3$-degenerate and $\epsilon$-flexibly $4$-choosable for some constant $\epsilon > 0$ \cite{DMMP-triangle-free}.
In addition,  the following theorem of Dvo\v{r}\'ak, Masa\v{r}\'ik, Mus\'ilek, and Pangr\'ac answers Question \ref{q:degen}
for the $3$-degenerate class of planar graphs of girth at least $6$:
\begin{theorem}[\cite{DMMP6}]
\label{thm:g6}
    There exists a constant $\epsilon > 0$ such that every planar graph of girth at least $6$ is weighted $\epsilon$-flexibly $3$-choosable.
\end{theorem}

One natural graph class for which to ask Question \ref{q:degen} is the class of graphs with bounded \emph{maximum average degree}, defined as follows. Given a graph $G$, the maximum average degree of $G$, written $\mad(G)$,
is the maximum value $\frac{2|E(H)|}{|V(H)|}$, where the maximum is taken over all nonempty subgraphs $H \subseteq G$.
We note that if $G$ is a graph with maximum average degree less than some integer $k$, then every nonempty subgraph of $G$ has a vertex of degree at most $k-1$, so $G$ is $(k-1)$-degenerate and $k$-choosable. Therefore, it is natural to ask
the following special case of
Question \ref{q:degen}:
\begin{question}
\label{q:mad}
    Let $k \geq 2$ be an integer. Does there exist a value $\epsilon = \epsilon(d) > 0$ such that every graph $G$ satisfying $\mad(G) < k$ is $\epsilon$-flexibly $k$-choosable?
\end{question}
We observe that since a planar graph of girth at least six has maximum average degree less than $3$, a positive answer to Question \ref{q:mad}, even for the special case $k=3$,
would imply Theorem \ref{thm:g6}.
Dvo\v{r}\'ak, Norin, and Postle
\cite[Lemma 12]{DNP} give the following
partial answer to Question  \ref{q:mad}:
\begin{theorem} 
\label{thm:2d3}
\cite{DNP}
If $G$ is a graph with maximum average degree less than $k-1 + \frac{2}{k+2}$, then $G$ is $\epsilon$-flexibly $k$-choosable for some constant $\epsilon = \epsilon(k) > 0$. 
\end{theorem}
When $k = 3$, Theorem \ref{thm:2d3} states that every graph with maximum average degree less than $2.4$ is $\epsilon$-flexibly $3$-choosable for some constant $\epsilon > 0$.
\subsection{Background: Reducible subgraph framework}
In order to prove Theorem \ref{thm:2d3},
Dvo\v{r}\'ak, Norin, and Postle \cite{DNP} implicitly use a method involving reducible subgraphs.
In graph theory, a \emph{reducible subgraph} framework is a setting
commonly used to solve graph coloring problems.
In this framework, one argues that for all graphs $G$ in some subgraph-closed class $\mathcal G$,
$G$ 
has a coloring $\phi:V(G) \rightarrow \mathbb N$ satisfying some particular property $P$.
In order to prove this statement, one first considers a counterexample $G \in \mathcal  G$ 
with the minimum number of vertices. Then, one shows that $G$ contains a particular subgraph $H$, and since $G$ is a minimal counterexample, $G \setminus H$ has a coloring satisfying $P$. Finally, one argues that the coloring on $G \setminus H$ can be extended to a coloring on $G$ satisfying $P$, which contradicts the assumption that $G$ is a counterexample and thereby proves the statement. A subgraph $H$ which allows such an argument is called a \emph{reducible subgraph}, and to prove that every graph in $\mathcal G$ has a coloring satisfying $P$, 
it suffices to prove that every graph in $\mathcal G$ has a reducible subgraph. 
Perhaps the most famous example of a reducible subgraph framework is the one used to prove the Four Color Theorem \cite{AH, 4CT}. Similar frameworks have been frequently used to prove upper bounds on other graph coloring parameters, such as acyclic chromatic number \cite{AMS},
injective chromatic number \cite{CKY},
and
injective chromatic index \cite{KRX}.

Dvo\v{r}\'ak, Masa\v{r}\'ik, Mus\'ilek, and Pangr\'ac \cite{DMMP-triangle-free}
explicitly define and develop the reducible subgraph framework for flexible list coloring which implicitly appears in \cite{DNP}.
In the framework of this method, an induced subgraph $H$ of a graph $G$ with a $k$-assignment $L$ is reducible if $H$ can 
always be $L$-colored even after all vertices of $G \setminus H$ are colored, and if certain additional properties are satisfied. Dvo\v{r}\'ak, Masa\v{r}\'ik, Mus\'ilek, and Pangr\'ac \cite{DMMP-triangle-free}
prove that if every induced subgraph of $G$ has a reducible subgraph $H$ with at most $b$ vertices, then there exists a constant $\epsilon = \epsilon(k,b) > 0$ such that $G$ is $\epsilon$-flexibly $k$-choosable.

The proof of Theorem \ref{thm:2d3}
in the special case $k = 3$
essentially
argues that 
a vertex of degree $1$ is a reducible subgraph, as is a pair of adjacent vertices of degree $2$.
Then, given a graph $G$ with maximum average degree less than $k -1 + \frac{2}{k+2} = 2.4$
and a $3$-assignment $L$ on $G$,
a discharging argument shows that every induced subgraph of $G$ contains one of these reducible subgraphs, implying the result.

\subsection{Our results}
The first goal of this paper is to introduce a framework of reducible subgraphs that generalizes the framework of Dvo\v{r}\'ak, Masa\v{r}\'ik, Mus\'ilek, and Pangr\'ac \cite{DMMP-triangle-free}. 
In the framework of \cite{DMMP-triangle-free}, in order to prove that a graph $G$ is weighted $\epsilon$-flexibly $k$-choosable,
one must prove that every induced subgraph of $G$ has a reducible subgraph on boundedly many vertices.
In our generalized framework, however, we
allow reducible subgraphs on arbitrarily many vertices.
While modifications of the framework from \cite{DMMP-triangle-free} are common
when considering restricted graph classes such as graphs of large girth \cite{DMMP6} or graphs with cycle restrictions \cite{YangYang},
our framework is the first to allow arbitrarily large reducible subgraphs.
We will see that this modification is powerful and 
allows us to make structural arguments which are incompatible with previous frameworks.

The second goal of this paper is to prove the following result, which gives a partial answer to Question \ref{q:mad} and strengthens Theorem \ref{thm:g6}.

\begin{theorem}
    \label{thm:mad3}
    If $G$ is a graph with maximum average degree less than $3$, then $G$ is 
    weighted 
    $2^{-30}$-flexibly $3$-choosable.
\end{theorem}
Since $K_4$ is a graph of maximum average degree exactly $3$ which is not $3$-choosable, Theorem \ref{thm:mad3} is best possible. Furthermore,
as observed above,
every planar graph of girth at least $6$
has maximum average degree less than $3$, so Theorem \ref{thm:mad3} implies Theorem \ref{thm:g6}
and is in fact a stronger statement.

The paper is organized as follows. In Section \ref{sec:framework}, we introduce our new reducible subgraph framework. In Section \ref{sec:reducible}, we establish a variety of tools for identifying reducible subgraphs. The proofs of the results in Section \ref{sec:reducible} are often rather tedious, and an impatient reader can skip the proofs of Section \ref{sec:reducible} without missing the main ideas of the paper. In Section \ref{sec:discharging}, we use a discharging argument to prove Theorem \ref{thm:mad3}.

\section{A generalized reducible subgraph framework}
\label{sec:framework}
In this section, we introduce a generalized version of the reducible subgraph framework developed by 
Dvo\v{r}\'ak, Masa\v{r}\'ik, Mus\'ilek, and Pangr\'ac \cite{DMMP6}.
We rely heavily on this generalized framework to prove Theorem \ref{thm:mad3}.
The key development of our framework is that we allow reducible subgraphs to be arbitrarily large, whereas the previous framework required reducible subgraphs to have boundedly many vertices.

We first state a lemma which gives a sufficient condition for weighted flexibility and is key to our framework.
A straightforward argument involving expected value
proves the lemma.
\begin{lemma} \label{lem:prob} \cite{DNP}
    Let $G$ be a graph, and let $k$ be a positive integer. Suppose that for every $k$-assignment $L$ on $G$, there exists a probability distribution on $L$-colorings $\phi$ of $G$ such that for each vertex $v \in V(G)$ and color $c \in L(v)$, 
    \[\Pr(\phi(v) = c) \geq \epsilon.\]
    Then, 
    $G$ is weighted $\epsilon$-flexibly $k$-choosable.
\end{lemma}

Next, we establish a definition which is central to our framework. 
Let $G$ be a graph, and let $L$ be a list assignment on $G$. Given values $0 < \epsilon \leq \alpha$, we say that a
distribution on $L$-colorings $\phi$ of $G$ is a $(k,\epsilon,\alpha)$-distribution if the following hold:
\begin{itemize}
    \item For each vertex $v \in V(G)$ and color $c \in L(v)$, $ \Pr(\phi(v) = c) \geq \epsilon$;
    \item For each subset $U \subseteq V(G)$ of at most $k-2$ vertices, the probability that $\phi(u) \neq c$ for all $u \in U$ is at least $\alpha^{|U|}$.
\end{itemize}
In particular, if $L$ admits a $(k,\epsilon,\alpha)$-distribution on $G$, then each color $c \in L(v)$ has a probability of at least $\epsilon$ of being used to color $v$ and a probability of at least $\alpha$ of being avoided at $v$. 
%Furthermore, if every $k$-assignment $L$ on $G$ admits a $(k,\epsilon,\alpha)$-distribution, then Lemma \ref{lem:prob} implies that $G$ is $\epsilon$-flexibly $k$-choosable. 
The notion of a $(k,\epsilon,\alpha)$-distribution appears implicitly in the reducibility framework of Dvo\v{r}\'ak, Masa\v{r}\'ik, Mus\'ilek, and Pangr\'ac \cite{DMMP6}.
\begin{definition}
    Let $G$ be a graph, and let $H$ be an induced subgraph of $G$. Let $k \geq 2$ be an integer. We say that $H$ is a \emph{$(k,\epsilon,\alpha)$-reducible} subgraph of $G$ if there exists a nonempty vertex subset $S \subseteq V(H)$ such that the following holds for every $k$-assignment $L$ on $G$: If there exists a $(k,\epsilon,\alpha)$-distribution on $L$-colorings of $G \setminus S$, then there exists a $(k,\epsilon,\alpha)$-distribution on $L$-colorings of $G$. We say that such a set $S$ is a \emph{reduction set} of $H$. 
\end{definition}

%We establish some examples of $3$-reducible subgraphs.

The following lemma, which resembles \cite[Lemma 3]{DMMP6}, shows that in order to show that a graph is flexibly choosable, it is enough to show that every induced subgraph contains a reducible subgraph.

\begin{lemma}
\label{lem:main-k-red}
    Let $G$ be a graph. If for every $Z \subseteq V(G)$, the graph $G[Z]$ contains a $(k,\epsilon,\alpha)$-reducible subgraph, then $G$ is weighted $\epsilon$-flexibly $k$-choosable. 
\end{lemma}
\begin{proof}
    We prove the stronger claim that if for every nonempty $Z \subseteq V(G)$, the graph $G[Z]$ contains a $(k,\epsilon,\alpha)$-reducible subgraph, then $G$ has a $(k,\epsilon,\alpha)$-distribution for every $k$-assignment on $G$. Then, the lemma's conclusion follows from Lemma \ref{lem:prob}.

    For our base case, if $|V(G)| = 0$, then the 
    distribution which assigns the empty coloring to $G$ with probability $1$ is vacuously a $(k,\epsilon,\alpha)$-distribution.
    Hence, we assume that $|V(G)| \geq 1$.
    Let $L$ be a $k$-assignment on $G$.
    By our assumption, there exists an induced subgraph $H$ of $G$ which is $(k,\epsilon,\alpha)$-reducible. 
    
    Let $S \subseteq V(H)$ be a reduction set of $H$. We consider the graph $G \setminus S$. By our lemma's assumption, each induced subgraph of $G \setminus S$ contains a $(k,\epsilon,\alpha)$-reducible subgraph. Therefore, as $|V(G) \setminus S| < |V(G)|$, our induction hypothesis tells us that $G \setminus S$ has a $(k,\epsilon,\alpha)$-distribution on $L$-colorings of $G \setminus S$. Then, as $H$ is $(k,\epsilon,\alpha)$-reducible, it then follows by definition that $L$ admits a $(k,\epsilon,\alpha)$-distribution on $G$. This completes the proof.
\end{proof}

While Lemma \ref{lem:main-k-red} gives a sufficient condition for the flexible choosability of a graph $G$ in terms of reducible subgraphs of $G$, it is not yet clear how to prove that an induced subgraph of $G$ is reducible. Therefore,
we establish a sufficient condition for determining that a subgraph of a graph $G$ is $(k,\epsilon,\alpha)$-reducible. This sufficient condition is closely related to the definition of reducibility of Dvo\v{r}\'ak, Masa\v{r}\'ik, Mus\'ilek, and Pangr\'ac \cite{DMMP6}.
Given a graph $G$ with induced subgraph $H$, for each integer $k \geq 3$, we define the function $\ell_{H,k}:V(H) \rightarrow \mathbb Z$ so that \[\ell_{H,k}(v) = k - \deg_G(v) + \deg_H(v)\]
for each $v \in V(H)$. Note that if $L$ is a $k$-assignment on $G$
and an $L$-coloring of $G \setminus H$ is fixed, then for each vertex $v \in V(H)$, $\ell_{H,k}(v)$ gives a lower bound for the number of available colors in $L(v)$.

%Now, we state our generalized reducible subgraph definition.

\begin{definition}
\label{def:fixforb}
Let $H$ be a graph, let $k \geq 3$ an integer, 
and let $f:V(H) \rightarrow \mathbb N$. Let $0 < \alpha \leq \frac{1}{k}$.
    We say that $H$ is \emph{$(f,k,\alpha)$-reductive} if 
 for every $f$-assignment $L$ on $H$, there exists a probability distribution on $L$-colorings $\phi$ of $H$ such that the following hold:
\begin{enumerate}
    \item[(FIX)] For each $v \in V(H)$ and each color $c \in L(v)$, $\Pr(\phi(v) = c) \geq \alpha$;
    \item[(FORB)] For each subset $U \subseteq V(H)$ of at most $k - 2$ vertices and each color $c \in \bigcup_{u \in U} L(u)$, $\Pr(\phi(u) \neq c \ \forall u \in U) \geq \alpha$.
\end{enumerate}
\end{definition}

Note that the existence of a probability distribution on proper $L$-colorings $\phi$ of $H$ implies that there exists a set $\Phi$ of proper $L$-colorings $\phi$ of $H$ with probability measure $1$. In particular, $\Phi$ is nonempty, so $H$ is $L$-colorable. Therefore, $H$ is $(f,k,\alpha)$-reductive only if $H$ is $f$-choosable.
Furthermore, by the (FORB) condition, if $H$ is $(f,k,\alpha)$-reductive, then $f(v) \geq 2$ for each vertex $v \in V(H)$.

We 
observe that when $k = 3$,
(FIX)
implies (FORB) whenever $f(v) \geq 2$ for each vertex $v \in V(H)$.
Therefore, in order to show that an induced subgraph $H$ of a graph $G$ is
$(f,3,\alpha)$-reductive,
it is enough to check that (FIX) holds for $H$ and that $f(v) \geq 2$ for each $v \in V(H)$.

For our next tool, we need the following probabilistic lemma.
\begin{lemma}
\label{lem:conditional}
    Let $A_1, \dots, A_t$ be disjoint events in a probability space with nonzero probability. Then, for each event $X$,
    \[\Pr(X|A_1 \cup \dots \cup A_t) = \sum_{i=1}^t \Pr(X|A_i) P(A_i | A_1 \cup \dots \cup A_t).\]
\end{lemma}
\begin{proof}
    By the definition of conditional probability,
    \begin{eqnarray*}
    \sum_{i=1}^t \Pr(X|A_i) P(A_i | A_1 \cup \dots \cup A_t)& =& \sum_{i=1}^t \frac{\Pr (X \cap A_i)}{\Pr(A_i)} \cdot \frac{\Pr(A_i)}{\Pr(A_1 \cup \dots \cup A_t)}  \\
    &=& \frac{\Pr(X \cap (A_1 \cup \dots \cup A_t) )}{\Pr(A_1 \cup \dots \cup A_t)} \\
    &=& \Pr(X | A_1 \cup \dots \cup A_t).
    \end{eqnarray*}
\end{proof}

The following lemma, which is is very similar to \cite[Lemma 4]{DMMP-triangle-free}, shows that under certain conditions, a subgraph of $G$ which is reductive is a reducible subgraph.

\begin{lemma}
    Let $G$ be a graph, and let $H$ be an induced subgraph of $G$. If $H$ is $(\ell_{H,k},k,\alpha)$-reductive subgraph of $G$, then $H$ is $(k,\epsilon,\alpha)$-reducible for each value $0 < \epsilon \leq (\frac{2\alpha}{k})^{k-1}$.
\end{lemma}
\begin{proof}
We show that $H$ is $(k,\epsilon,\alpha)$-reducible with a reduction set $S = V(H)$. 
We let $L$ be a $k$-assignment on $G$. We assume that $G \setminus H$ has a $(k,\epsilon,\alpha)$-distribution on $L$-colorings, and we aim to show that $G$ has a $(k,\epsilon,\alpha)$-distribution on $L$-colorings.

We construct a distribution on $L$-colorings of $G$ as follows. First, we randomly choose an $L$-coloring $\phi$ on $G \setminus H$ according to a $(k,\epsilon,\alpha)$-distribution. Then, we let $L'$ be the list assignment for $H$ defined by $L'(z) = L(z) \setminus \{\phi(v) : v \in V(G \setminus H) \cap N(z)\}$, for each $z \in  V(H)$.
    In other words $L'(z)$ consists of the colors from $L(z)$ which are available at $z$ after $G \setminus H$ is colored by $\phi$.
    Note that $|L'(z)| \geq \ell_{H,k}(z) \geq 2$ for each $z \in V(H)$.
    We choose a set of $ \ell_{H,k}(z)$ colors uniformly at random from each list $L'(z)$, 
    and we delete all other colors from $L'(z)$, so that $|L'(z)| = \ell_{H,k}(z)$ for each vertex $z \in V(H)$.
    Then, we define a probability distribution on $L'$-colorings of $H$ that satisfies (FIX) and (FORB) for our value $\alpha$, and we choose an $L'$-coloring $\psi$ of $H$ according to this distribution. Finally, we combine $\phi$ and $\psi$ to obtain an $L$-coloring of $G$.

    We first argue that for each vertex $v \in V(G)$ and color $c \in L(v)$, the probability that $v$ is colored with $c$ is at least $\epsilon$.     
    If $v \in V(G) \setminus V(H)$ and $c \in L(v)$, then by the induction hypothesis, $\Pr(\phi(v) = c) \geq \epsilon$. If $v \in V(H)$ and $c \in L(v)$,
    then let $U$ be the set of neighbors of $v$ in $V(G) \setminus V(H)$. 
    Since $H$ is $(\ell_{H,k},k, \alpha)$-reductive, it follows that $\ell_{H,k}(v) \geq 2$; thus, 
    $v$ has at most $k-2$ neighbors in $V(G) \setminus V(H)$, so $|U| \leq k-2$.
    Hence, as $\phi$ is chosen according to a $(k,\epsilon,\alpha)$-distribution, the probability that $\phi(u) \neq c$ for all vertices $u \in U$ is at least $\alpha^{k-2}$.

Next, suppose it is given that $\phi = \phi_0$, where $\phi_0$ is a fixed $L$-coloring of $G \setminus H$ such that $\phi_0(u) \neq c$ for all $u \in U$. With $\phi = \phi_0$ given, the conditional probability that $c \in L'(v)$ is at least $\frac{\ell_{H,k}(v)}{k} \geq \frac{2}{k}$. Then, as our distribution on colorings $\phi$ satisfies (FIX) and (FORB),
the subsequent conditional probability that $\psi(v) = c$ is at least $\alpha$. Therefore, with $\phi = \phi_0$ given, the conditional probability that $\psi(v) = c$ is at least $\frac{2\alpha}{k}$.

Now, let $\Phi$ be the set of all fixed $L$-colorings $\phi_0$ of $G \setminus H$ for which $\phi_0(u) \neq c$ for all $u \in U$. By Lemma \ref{lem:conditional}, 
\begin{eqnarray*}
\Pr (\psi(v) = c | \phi(u) \neq c \  \forall u \in U)  &=&
\Pr \left (\psi(v) = c \biggr \rvert \bigcup_{\phi_0 \in \Phi} (\phi = \phi_0) \right ) \\
&=& \sum_{\phi_0 \in \Phi} \Pr(\psi(v) = c| \phi = \phi_0) \Pr(\phi = \phi_0 | \phi \in \Phi) \\
&\geq & \frac{2\alpha}{k} \sum_{\phi_0 \in \Phi} \Pr(\phi = \phi_0 | \phi \in \Phi) \\
&=& \frac{2\alpha}{k}.
\end{eqnarray*}
Therefore, $v$ ultimately receives the color $c$ with probability at least $(\frac{2\alpha}{k})^{k-2} \left ( \frac{2 \alpha}{k} \right )= \epsilon$.

    Next, we argue that if $U \subseteq V(G)$ is a vertex subset of size at most $k-2$ and $c \in \bigcup_{u \in U} L(u)$,
    then with probability at least $\alpha^{|U|}$, no vertex $u \in U$ receives the color $c$. 
    We write $U_1 = U \setminus V(H)$ and $U_2 = U \cap V(H)$. If $U = U_1$, then by the induction hypothesis, $\phi(u) \neq c$ for all $u \in U$ with probability at least $\alpha^{|U|}$. Otherwise, 
    the induction hypothesis tells us that $\phi(u) \neq c$ for all $u \in U_1$ with probability at least $\alpha^{|U_1|}$.
Next, suppose it is given that $\phi = \phi_0$, where $\phi_0$ is a fixed $L$-coloring of $G \setminus H$ such that $\phi_0(u) \neq c$ for all $u \in U_1$. With $\phi = \phi_0$ given, by (FORB), the conditional probability that $\psi(u) \neq c$ for all $u \in U_2$ is at least $\alpha$.

Now, let $\Phi$ 
be the set of all fixed $L$-colorings $\phi_0$ of $G \setminus H$ for which $\phi_0(u) \neq c$ for all $u \in U_1$. By Lemma \ref{lem:conditional},
\begin{eqnarray*}
   & &  \Pr\left (\psi(u) \neq c \ \forall u \in U_2 \biggr |  \phi(u) \neq c \ \forall u \in U_1 \right ) \\
    &=& \Pr \left (\psi(u) \neq c \ \forall u \in U_2 | \bigcup_{\phi_0 \in \Phi} (\phi = \phi_0) \right ) \\
    &=& \sum_{\phi_0 \in \Phi}  \Pr \left (\psi(u) \neq c \ \forall u \in U_2 | (\phi = \phi_0) \right )  \Pr(\phi = \phi_0 | \phi \in \Phi) \\
    &\geq & \alpha \sum_{\phi_0 \in \Phi}    \Pr(\phi = \phi_0 | \phi \in \Phi) \\
    &=& \alpha.
\end{eqnarray*}
As $\phi(u) \neq c$ for all $u \in U_1$ with probability at least $\alpha^{|U_1|}$, our final coloring $\phi \cup \psi$ avoids $c$ at all $u \in U$ with probability at least $\alpha \cdot \alpha^{|U_1|} \geq \alpha^{|U_2| + |U_1|} = \alpha^{|U|}$. Therefore, our distribution on $L$-colorings of $G$ is a $(k,\epsilon,\alpha)$-distribution, completing the proof.   
\end{proof}

In the previous framework of Dvo\v{r}\'ak, Masa\v{r}\'ik, Mus\'ilek, and Pangr\'ac \cite{DMMP6}, 
 a subgraph $H$ of $G$ is $k$-reducible 
if for every $\ell_{H,k}$-assignment $L$ on $H$,
the probabilities in (FIX) and (FORB) are positive, but not necessarily bounded below by some $\alpha$.
However, in practice, the previous framework only considers reducible subgraphs with at most $b$ vertices, for some constant $b$.
Therefore, if $H$ is $k$-reducible in the previous framework and $|V(H)| \leq b$, then
a uniform distribution on all $L$-colorings of $H$ guarantees that (FIX) and (FORB) both hold with the value $\alpha = k^{-b}$;
hence, $H$ is also $(k,\epsilon, \alpha)$-reducible in our new framework whenever $\alpha = k^{-b}$ and $0 < \epsilon \leq  (\frac{2\alpha} {k})^{k-1}$.

\section{Tools for identifying reducible subgraphs}
\label{sec:reducible}
In order to prove Theorem \ref{thm:mad3},
we consider a graph $G$ of maximum average degree less than $3$, and we show that every induced 
subgraph of $G$ has an induced subgraph which is $(3,\epsilon,\alpha)$-reducible for some constants $\\epsilon, alpha > 0$. 
Then, we apply Lemma \ref{lem:main-k-red} to argue that $G$ is weighted $\epsilon$-flexibly $3$-choosable.
In order to prove that an induced subgraph $H$ of $G$ is $(3,\epsilon,\alpha)$-reducible, we need tools for constructing list coloring distributions on $H$. In this section, we aim to develop those tools.
%We give names to the lemmas in this section that we frequently need later in order to make them easier to reference.

%If $G$ is a graph with a list assignment $L$, $H$ is an induced subgraph of $G$, and $\phi$ is an $L$-coloring of $G\setminus H$, then we say that a color $c \in L(v)$ is \emph{available} at a vertex $v \in V(H)$ if $\phi(w) \neq c$ for each neighbor $w \in N(v) \setminus V(H)$.
Given a connected graph $G$, a vertex $v \in V(G)$ is a \emph{cut vertex} if $G \setminus \{v\}$ has at least two components.
A connected induced subgraph $B$ of $G$ is a \emph{block} if the graph $B$ has no cut vertex and every connected subgraph $H$ of $G$ satisfying $B \subsetneq H \subseteq G$ has a cut vertex.
In other words,
an
induced subgraph $B \subseteq G$ is a block if $B$ is maximal with respect to the property of being $2$-connected or isomorphic to $K_2$. 
A \emph{terminal block} of $G$ is a block containing at most one cut vertex of $G$.
The \emph{block-cut tree} of $G$ is a tree $T$ whose vertices consist of the blocks and cut vertices of $G$, where a block $B$ is adjacent with a cut vertex $v \in V(G)$ if and only if $v \in V(B)$. Note that a block $B$ of $G$ is a terminal block if and only if $B$ has degree $1$ or $0$ in the block-cut tree of $G$.

\begin{lemma}
\label{lem:perfect}
    Let $G$ be a graph, and let $H$ be an induced subgraph of $G$. If $H$ is a terminal block of $G$ and $H$ is $(3,3,\frac 13)$-reductive, then $H$ is $(3,\epsilon,\alpha)$-reducible for all values $0 < \epsilon \leq \alpha \leq \frac{1}{3}$.
\end{lemma}
\begin{proof}
If $H = G$, then $H$ is $(3,3,\frac{1}{3})$-reductive; thus, for each $3$-assignment $L$ on $G$, there exists a distribution on $L$-colorings $\phi$ of $G$ so that for each vertex $v \in V(G)$ and color $c \in L(v)$, $\Pr(\phi(v) = c) = \frac{1}{3}$.  Hence, $L$ admits a $(3,\epsilon,\alpha)$-distribution on $G$.

Otherwise, as $H$ is a terminal block of $G$, $V(G) \cap V(H)$ contains a single cut vertex $x$. 
We show that $H$ is $(3,\epsilon,\alpha)$-reducible with a reduction set $S = V(H) \setminus \{x\}$.
Let $L$ be a $3$-assignment on $G$, and suppose that there exists a $(3,\epsilon,\alpha)$-distribution on $L$-colorings of $G \setminus S$. 
We construct a $(3,\epsilon,\alpha)$-distribution on $L$-colorings of $G$ as follows. First, we randomly choose an $L$-coloring $\phi$ of $G \setminus S$ according to a $(3,\epsilon,\alpha)$-distribution, and we use $\phi$ to color $G \setminus S$.
Next, we fix a distribution on $L$-colorings $\psi$ of $H$ such that for each $v \in V(H)$ and $c \in L(v)$, $\Pr(\psi(v) = c) = \frac{1}{3}$. 
Finally, we give $H$ an $L$-coloring according to the conditional random variable $\psi|(\psi(x) = \phi(x))$. We combine these colorings of $G \setminus S$ and $H$ in order to obtain a random $L$-coloring of $G$.

 We first argue that for each vertex $v \in V(G)$ and color $c \in L(v)$, 
 $v$ receives the color $c$ with probability at least $\epsilon$.
 If $v \in V(G) \setminus S$ and $c \in L(v)$, then by the induction hypothesis, $\Pr(\phi(v) = c) \geq \epsilon$. If $v \in S$ and $c \in L(v)$,
    then by Lemma \ref{lem:conditional}, $\psi$ assigns $c$ to $v$ with probability
    \begin{eqnarray*}
    \Pr(\psi(v) = c | \psi(x) = \phi(x)) &=&\sum_{c' \in L(x)} \Pr(\psi(v) = c | \psi(x) = c') \Pr(\phi(x) = c') \\
    &\geq & \varepsilon \sum_{c' \in L(x)} \frac{\Pr(\psi(v) = c \land \psi(x) = c')}{\Pr(\psi(x) = c')} \\
    &=& 3 \epsilon \sum_{c' \in L(x)} \Pr(\psi(v) = c \land \psi(x) = c') \\
    &=&
    3 \epsilon \Pr(\psi(v) = c) = \epsilon.
    \end{eqnarray*}
    Therefore, $v$ is colored with $c$ with probability at least $\epsilon$.
    
    Next, we aim to show that with probability at least $\alpha$, $v$ does not receive the color $c$.    
    If $v \not \in S$, then by the induction hypothesis, 
    $\phi(v) \neq c$ with probability at least $\alpha$, and the argument is complete. 
    Otherwise, $v \in S$. Then, with $\phi$ fixed, the conditional probability that $v$ is colored with $c$ is
    \begin{eqnarray*}
    \Pr(\psi(v) = c  \ | \psi(x) = \phi(x)) &=&\sum_{c' \in L(x)}  \Pr(\psi(v) =c | \psi(x) = c') \Pr(\phi(x) = c') \\
    &\leq & (1-\alpha) \sum_{c' \in L(x)} \frac{ \Pr(\psi(v) = c   \land \psi(x) = c') }{\Pr(\psi(x) = c')} \\
    &=& 3(1-\alpha) \sum_{c' \in L(x)} \Pr(\psi(v) = c   \land \psi(x) = c')  \\
    &=&
     3(1-\alpha) \Pr(\psi(v) = c  ) \\
    &=& 1 - \alpha.
    \end{eqnarray*}
    Therefore, our distribution on $L$-colorings of $G$ is a $(3,\epsilon,\alpha)$-distribution.    
\end{proof}

We say that a \emph{diamond} is a graph on exactly four vertices with exactly five edges. 
We write $K_4^-$ for the graph isomorphic to a diamond.
In other words, a diamond is a graph obtained from $K_4$ by deleting one edge.
The following lemma implies that if a graph $G$
has a terminal block $B$ isomorphic to a diamond, then $B$ is a $(3,\epsilon,\alpha)$-reducible subgraph of $G$ for all $0 < \epsilon \leq \alpha \leq \frac{1}{3}$.

\begin{lemma}
\label{lem:diamond}
    If $D$ is a diamond, then $D$ is $(3,\epsilon,\alpha)$-reducible for all values $0 < \epsilon \leq \alpha \leq \frac{1}{3}$.
\end{lemma}
\begin{proof}
    Let $L$ be a $3$-assignment on $D$.
    As $D$ has treewidth $2$, 
    it follows from \cite[Theorem 3.2]{BMS} that there exists a set $\Phi$ of six $L$-colorings of $D$ such that for each vertex $v \in V(D)$ and color $c \in L(v)$, $\phi(v) = c$ for exactly two $L$-colorings $\phi \in \Phi$. Therefore, by taking an $L$-coloring from $\Phi$ uniformly at random, we see that $H$ is $(3,3,\frac{1}{3})$-reductive. Then, the lemma follows from Lemma \ref{lem:perfect}.
\end{proof}

Our next lemma, first introduced by Erd\H{o}s, Rubin, and Taylor \cite{ERT}, characterizes the graphs $G$ which have an $L$-coloring for every list assignment $L$ that gives at least $\deg(v)$ colors to each vertex $v \in V(G)$.
In the lemma statement, a \emph{theta} is a graph obtained from a cycle by adding a single edge.

\begin{lemma}
\label{lem:ERT}
If $G$ is a connected
graph and $f:V(G) \rightarrow \mathbb N$ is a function satisfying $f(v) \geq \deg(v)$ for each $v \in V(G)$, then $G$ is $f$-choosable if and only if one of the following conditions holds:
\begin{itemize}
    \item $f(v) > \deg(v)$ for at least one $v \in V(G)$;
    \item $G$ has a block that is not a clique and is not an odd cycle;
   \item $G$ has an induced even cycle or an induced theta subgraph.
\end{itemize}
\end{lemma}
Erd\H{o}s, Rubin, and Taylor also proved that the second and third conditions of Lemma \ref{lem:ERT} are equivalent.
Lemma \ref{lem:ERT} implies that if $G$ is a connected graph and $f(v) = \deg(v)$ for each vertex $v \in V(G)$, then the only case in which $G$ is not $f$-choosable occurs when each block of $G$ is a clique or odd cycle. A connected graph in which each block is a clique or odd cycle is called a \emph{Gallai tree}. We need one more sufficient condition by which a Gallai tree is $L$-colorable.

\begin{lemma}[{\cite{KSW}}]
    \label{lem:Gallai}
    Let $G$ be a Gallai tree, and let $L$ be a list assignment on $G$ satisfying $|L(v)| = \deg(v)$ for each vertex $v \in V(G)$. If $G$ has a terminal block $B$ with two vertices $u,w \in V(B)$ for which $\deg(u) = \deg(w)$ and $L(u) \neq L(w)$, then $G$ is $L$-colorable. 
\end{lemma}

These two lemmas allow us to show that the graphs $H_5$ and $H_7$ in Figure \ref{fig:exceptions} are $(3,\epsilon,\alpha)$-reducible subgraphs under certain reasonable conditions.

\begin{lemma}
\label{lem:H5}
Let $\alpha \leq \frac{1}{10}$ and $\epsilon \leq \frac{1}{15}\alpha$. Suppose that
the graph $H_5$ in Figure \ref{fig:exceptions} is a subgraph of a graph $G$ 
for which each $s \in V(H_5)$ satisfies $\deg_G(s) = 3$. Then,
$H_5$ is a $(3,\epsilon,\alpha)$-reducible subgraph of $G$.
\end{lemma}
\begin{proof}
    We show that if each vertex $s \in V(H_5)$ has degree $3$ in $G$, then $H_5$ is a $(3,\epsilon,\alpha)$-reducible subgraph of $G$ with a reduction set $S = V(H_5)$. 
    Let $L$ be a $3$-assignment on $V(G)$, and suppose that $G \setminus H_5$ has a $(3,\epsilon,\alpha)$-distribution on $L$-colorings. 
    We randomly choose an $L$-coloring $\phi$ of $G \setminus H_5$ according to this distribution, and we extend $\phi$ to all of $V(G)$ as follows.
    We write $x'$ for the unique neighbor of $x$ in $G \setminus H$.
    We let $L'(x) = L(x) \setminus \phi(x')$, and we let $L'(s) = L(s)$ for each $s \in V(H_5) \setminus \{x\}$.
    Next, for each vertex $s \in V(H_5)$, we define a subset $L''(s) \subseteq L'(s)$ as follows. If there exist colors $a_1, a_2, a_3, a_4$ so that $L(v) = L(z) = \{a_1, a_2, a_4\}$, $L(y) = \{a_1, a_2, a_3\}$, and $L'(x) = \{a_3,a_4\}$, then we define $L''(w) = L'(w) \setminus \{a_4\}$. Otherwise, we define $L''(w) = L'(w)$.
    Symmetrically, if there exist colors $a_1, a_2, a_3, a_4$ so that $L(v) = L(z) = \{a_1, a_2,a_4\}$, $L(w) = \{a_1,a_2,a_3\}$, and $L'(x) = \{a_3,a_4\}$, then we define $L''(y) = L'(y) \setminus \{a_4\}$. Otherwise, we let $L''(y) = L'(y)$. For all other vertices $s \in V(H_5)$, we let $L''(s) = L'(s)$.

    Now, we choose a vertex $u \in V(H_5)$ uniformly at random and a color $c \in L''(u)$ uniformly at random, and we assign $\phi(u) = c$. We then delete the color $c$ from $L'(s)$ for all neighbors $s$ of $u$ in $H_5$. We claim that $H' = H_5 \setminus \{u\}$ is $L'$-colorable, and we finish extending our coloring $\phi$ to all of $G$ by assigning an arbitrary $L'$-coloring to $H'$. To show that $H'$ is $L'$-colorable, we consider several cases.
    \begin{enumerate}
        \item If $u \in \{v,x,z\}$, then after coloring $u$, $H' = H_5 \setminus \{u\}$ is isomorphic either to a diamond or a $C_4$. Furthermore, 
        $|L'(s)| \geq \deg_{H'}(s)$ for
        each vertex $s \in V(H')$, and each color in $L'(s)$ is available at $s$. Therefore, $H'$ is $L'$-colorable by Lemma \ref{lem:ERT}.
        \item If $u = w$, then Lemmas \ref{lem:ERT} and \ref{lem:Gallai} imply that $H'$ is not $L'$-colorable if and only if there exist colors $a_1, a_2, a_3$ for which $L'(v) = L'(z) = \{a_1,a_2\}$, $L'(y) = \{a_1,a_2,a_3\}$, and $L'(x) = \{a_3\}$. If this is the case, then $\phi(w) = a_4$ for some color $a_4$, and  $L'(v) = L'(z) = \{a_1, a_2, a_4\}$, $L'(y) = \{a_1, a_2, a_3\}$, and $L'(x) = \{a_3, a_4\}$. However, in this special case, $a_4 \not \in L''(w)$, contradicting the choice of $\phi(w) = a_4$. Therefore $H'$ is $L'$-colorable.
        \item If $u = y$, then this case is symmetric to the previous case.
    \end{enumerate}

    Now, we claim that for each $s \in V(H_5)$ and $c \in L(s)$, $\Pr(\phi(s) = c) \geq \frac{1}{15}\alpha$. 
    If $s \in \{v,z\}$, then the probability that $s$ is chosen to be colored first is $\frac{1}{5}$. Furthermore, $L''(s) = L(s)$, so the color $c$ is chosen with subsequent probability $\frac{1}{3}$, for an overall probability of $\frac{1}{15}$. If $s = x$, then $x$ is colored first with probability $\frac{1}{5}$. Then, $c \in L''(x)$ whenever $\phi(x') \neq c$, which occurs with probability at least $\alpha$. Then, $c$ is chosen from $L''(x)$ with subsequent probability at least $\frac{1}{3}$, for a total probability of at least $\frac{1}{15}\alpha$. If $s = w$, then $w$ is colored first with probability $\frac{1}{5}$. Next, $c \not \in L''(w)$ if and only if there exist colors $a_1, a_2, a_3, c$ such that $L'(v) = L'(z) = \{a_1, a_2, c\}$, $L'(y) = \{a_1,a_2,a_3\}$, and $L'(x) = \{a_3, c\}$. We observe that if $L'(v)$, $L'(z)$, and $L'(y)$ are fixed as above, then $c \in L''(w)$ whenever $L'(x) \neq \{a_3, c\}$.
    Hence, letting $a \in L(x) \setminus \{a_3, c\}$ be a fixed color, we see that $c \in L''(w)$ whenever $\phi(x') \neq a$. Therefore, $c \in L''(w)$ with probability at least $\alpha$. Then, $c$ is chosen from $L''(w)$ with subsequent probability at least $\frac{1}{3}$. Hence, the overall probability that $\phi(w) = c$ is at least $\frac{1}{15}\alpha$. Finally, if $s = y$, then by applying the same argument used with $w$, $\phi(y) = c$ with probability at least $\frac{1}{15}\alpha$.

    Finally, we consider a vertex $s \in V(H_5)$ and a color $c \in L(s)$, and we estimate the probability that $\phi(s) \neq c$. For the event $\phi(s) \neq c$ to occur, it is sufficient that $s$ is colored first and a color other than $c$ is chosen from $L''(s)$. Since $|L''(s)| \geq 2$, $\phi(s) \neq c$ with probability at least $\frac{1}{10}$. Hence, our randomly constructed $L$-colorings $\phi$ of $G$ form a $(3, \epsilon, \alpha)$-distribution, and hence $H$ is $(3,\epsilon,\alpha)$-reducible.
\end{proof}

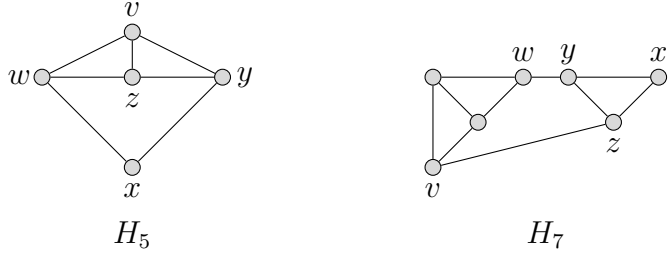
\begin{figure}
\begin{center}
\begin{tikzpicture}
[scale=1.2,auto=left,every node/.style={circle,fill=gray!30,minimum size = 6pt,inner sep=0pt}]
\node(z) at (0,-1.75) [draw=white,fill=white] {$H_5$};
\node(z) at (3,-1.25) [draw=white,fill=white] {};
\node(z) at (0,-1.25) [draw=white,fill=white] {$x$};
\node(z) at (-1.25,0) [draw=white,fill=white] {$w$};
\node(z) at (1.25,0) [draw=white,fill=white] {$y$};
\node(z) at (0,-0.25) [draw=white,fill=white] {$z$};
\node(z) at (0,0.75) [draw=white,fill=white] {$v$};
\node(p1) at (-1,0) [draw = black] {};
\node(p2) at (0,0) [draw = black] {};
\node(p3) at (1,0) [draw = black] {};
\node(p4) at (0,0.5) [draw = black] {};
\node(r) at (0,-1) [draw=black] {};
%\node(h1) at (-1,0) [draw=black] {};
%\node(h2) at (1,0) [draw = black] {};
\foreach \from/\to in {p1/p2,p2/p4,p2/p3,p3/p4,r/p1,p4/p1,r/p3}
    \draw (\from) -- (\to);
\end{tikzpicture}
\begin{tikzpicture}
[scale=1.2,auto=left,every node/.style={circle,fill=gray!30,minimum size = 6pt,inner sep=0pt}]
\node(z) at (0.75,-1.75) [draw=white,fill=white] {$H_7$};
\node(z) at (-0.5,-1.25) [draw=white,fill=white] {$v$};
\node(z) at (0.5,0.25) [draw=white,fill=white] {$w$};
\node(z) at (1,0.25) [draw=white,fill=white] {$y$};
\node(z) at (1.5,-0.75) [draw=white,fill=white] {$z$};
\node(z) at (2,0.25) [draw=white,fill=white] {$x$};
\node(p1) at (-0.5,0) [draw = black] {};
\node(p2) at (0,-0.5) [draw = black] {};
\node(p3) at (0.5,0) [draw = black] {};
\node(p4) at (1,0) [draw = black] {};
\node(p5) at (2,0) [draw = black] {};
\node(p6) at (1.5,-0.5) [draw = black] {};
\node(r) at (-0.5,-1) [draw=black] {};
%\node(h1) at (-1,0) [draw=black] {};
%\node(h2) at (1,0) [draw = black] {};
\foreach \from/\to in {p1/p2,p1/p3,p2/p3,p3/p4,p4/p5,p5/p6,p4/p6,r/p1,r/p2,r/p6}
    \draw (\from) -- (\to);
\end{tikzpicture}
\end{center}
\caption{The figure shows the graphs $H_5$ and $H_7$, which are $(3,\epsilon,\alpha)$-reducible under certain conditions.}
\label{fig:exceptions}
\end{figure}

\begin{lemma}
\label{lem:H7}
Let $\alpha \leq \frac{1}{14}$ and $\epsilon \leq \frac{1}{21}\alpha$. Suppose that
the graph $H_7$ in Figure \ref{fig:exceptions} is a subgraph of a graph $G$ 
such that each $s \in V(H_7)$ satisfies $\deg_G(s) = 3$. Then,
$H_7$ is a $(3,\epsilon,\alpha)$-reducible subgraph of $G$.
\end{lemma}
\begin{proof}
    We show that $H_7$ is a reducible subgraph of $G$ with a reduction set $S = V(H_7)$.
    Let $L$ be a $3$-assignment on $V(G)$, and suppose that $G \setminus H_7$ has a $(3,\epsilon,\alpha)$-distribution on $L$-colorings. We choose an $L$-coloring $\phi$ of $G \setminus H_7$ according to this distribution,
    and we extend $\phi$ to all of $V(G)$ as follows. 
    We write $x'$ for the unique neighbor of $x$ in $G \setminus H$.
    We let $L'(x) = L(x) \setminus \phi(x')$, and we let $L'(s) = L(s)$ for each $s \in V(H_5) \setminus \{x\}$.
    Next, we define a set $L''(s) \subseteq L'(s)$ as follows. If  $|L'(x)| = 2$ and $L'(v)$ contains a color $a$ such that $L'(z) \setminus \{a\} = L'(x)$, %{\RB should we specify $|L'(x)| = 2$} {\PB Yes},
    then we define $L''(v) = L'(v) \setminus \{a\}$. Otherwise, we define $L''(v) = L'(v)$.
    Similarly, if  $|L'(x)| = 2$ and $L'(w)$ contains a color $a$ for which $L'(y) \setminus \{a\}  = L'(x)$, then we define $L''(w) = L'(w) \setminus \{a\}$. Otherwise, we define $L''(w) = L'(w)$.
    We write $L''(s) = L'(s)$ for each other vertex $s \in V(H_7)$.
    Now, we choose a vertex $u \in V(H_7)$ uniformly at random, and then we choose a color $c \in L''(u)$ uniformly at random and assign $\phi(u) = c$. We delete $c$ from $L'(s)$ for all neighbors $s$ of $u$ in $H_7$.
    We define $H' = H_7 \setminus \{u\}$, and we observe that for each $s \in V(H')$, $|L'(s)| \geq \deg_{H'}(s)$. We claim that $H'$ is $L'$-colorable by considering the following cases.
    \begin{enumerate}
        \item If $u = v$, then by construction, $L'(z) \neq L'(x)$. Therefore, $H'$ is $L'$-colorable by Lemma \ref{lem:Gallai}.
        \item If $u = w$, then by construction, $L'(y) \neq L'(x)$. Therefore, $H'$ is $L'$-colorable by Lemma \ref{lem:Gallai}.
        \item In all other cases, $H'$ contains an induced 
        theta subgraph, so $H'$ is $L'$-colorable by Lemma \ref{lem:ERT}.
    \end{enumerate}
    Now, given a vertex $s \in V(H_7)$ and a color $c \in L(s)$, we estimate the probabilities that $\phi(s) = c$ and $\phi(s) \neq c$. 
    If $s \not \in \{v,w,x\}$, then $c \in L''(s) = L(s)$; hence, $\phi(u) = c$ in the event that $s$ is colored first and $c$ is chosen from $L''(s)$, which occurs with probability at least $\frac{1}{21}$. If $s = x$, then $c \in L''(x)$ whenever $\phi'(x') \neq c$. Then, $\phi(x) = c$ whenever $x$ is colored first and $c$ is chosen from $L''(x)$. Hence $\phi(x) = c$ with a probability of at least $\frac{1}{21}\alpha$.
    If $s = w$, then $c \not \in L''(w)$ if and only if $\phi(x')$ is the unique color $a$ for which $L(y) \setminus \{c\} = L(x) \setminus \{a\}$. Hence, $c \in L''(w)$ with probability at least $\alpha$, and hence using a similar argument, $\phi(w) = c$ with probability at least $\frac{1}{21}\alpha$.
    If $s = v$, then by using a symmetric argument, $\phi(v) = c$ with probability at least $\frac{1}{21} \alpha$.
    Furthermore, $\phi(s) \neq c$ whenever $s$ is colored first and $c$ is not chosen from $L''(s)$, which occurs with probability at least $\frac{1}{7} \cdot \frac{1}{2} = \frac{1}{14}$. Hence, our randomly constructed $L$-colorings of $G$ form a $(3,\epsilon,\alpha)$-distribution.
\end{proof}

The next lemma roughly shows that in order to check whether a graph $H$ is $(f,3,\alpha)$-reductive for some constant $\alpha > 0$, it is enough to  consider each terminal block of $H$ independently, as well as the graph obtained from $H$ by deleting each terminal block.

\begin{lemma}
    \label{lem:cut}
    Let $H$ be a graph, 
    and let $f:V(H) \rightarrow \mathbb N$ be a function.
    Let $\alpha, \beta > 0$. Suppose that the following hold:
    \begin{itemize}
        \item  $H$ is the union of subgraphs $H^*, H_1 \dots, H_t$;
        \item The sets $V(H_1 \setminus H^*), \ldots, V(H_t \setminus H^*)$ are pairwise disjoint and have no edges joining them;
        \item For $1 \leq i \leq t$, $V(H_i) \cap V(H^*)$ consists of a single vertex $v_i$; %that separates $H^* \setminus \{ v_i \}$ from $H_i \setminus \{ v_i\}$;
        \item For $1 \leq i \leq t$, $H_i$ is $(f,3,\alpha)$-reductive;
        \item $H^*$ is $(f,3,\beta)$-reductive.
    \end{itemize}
    %for each $f_{|H_i}$-assignment $L_i$ on $H_i$, there exists a probability distribution on $L_i$ colorings $\phi_i$ of $H_i$ such that for each $v \in V(H_i)$ and $c \in L(v)$, $\phi_i(v) = c$ with probability at least $\alpha$. 
    Then $H$ is $(f,3,\alpha \beta)$-reductive.
    %for every $f$-assignment $L$ on $G$, 
    %there exists a probability distribution on $L$-colorings $\phi$ of $G$ such that for each $v \in V(G)$ and $c \in L(v)$, the probability that $\phi(v) = c$ is at least $\alpha / 3$.
\end{lemma}
\begin{proof}
     We first observe that as all subgraphs $H^*$ and $H_i$ are $(f,3,\min\{\alpha,\beta\})$-reductive, it follows that $f(v) \geq 2$ for each $v \in V(H)$. Therefore, in order to show that $H$ is $(f,3, \alpha \beta)$-reductive,
     it suffices to show that for each $f$-assignment $L$ on $H$, we can find a probability distribution on $L$-colorings of $H$ satisfying the (FIX) condition.

    Consider an $f$-assignment $L$ on $V(H)$.
    As $H^*$ is $(f,3,\beta)$-reductive,
    we define a random variable $\phi^*$
    which assigns an
    $L$-coloring to $H^*$ 
    such that
    for each vertex $w \in V(H^*)$ and color $c \in L(w)$, $\Pr(\phi^*(w) = c) \geq \beta$. 
    For each value $1 \leq i \leq t$,  $H_i$ is $(f,3,\alpha)$-reductive,
    so we define a random variable $\phi_i$ 
    which assigns an $L$-coloring to $H_i$ such that for each vertex $w \in V(H_i)$ and color $c \in L(w)$, $\Pr(\phi_i(w) = c) \geq \alpha$. 

    Now, we randomly choose an $L$-coloring on $H$ as follows. First, we randomly choose an $L$-coloring $\phi^*$ of $H^*$.
    Next, for each subgraph $H_i \neq H^*$, 
    we assign $H_i$ an $L$-coloring
    using the conditional random variable $\phi_i | (\phi^*(v_i) = \phi_i(v_i))$.
    We then give $H$ an $L$-coloring by combining the colorings $\phi^*, \phi_1, \dots, \phi_t$.

    For each vertex $w \in V(H^*)$ and color $c' \in L(w)$, $\Pr(\phi^*(w) = c') \geq \beta \geq \alpha \beta$.
    Hence, we consider a subgraph $H_i$, a vertex $v \in V(H_i)$, and a color $c \in L(v)$. We observe that $c$ is assigned to $v$ with probability
\begin{eqnarray*}
    \Pr(\phi_i(v) = c | \phi^*(v_i) = \phi_i(v_i)) 
    & \geq & \Pr(\phi_i(v) = c \land \phi_i(v_i) = \phi^*(v_i)) \\ 
    &=& \Pr(\phi_i(v) = c) \Pr(\phi^*(v_i) = \phi_i(v_i)| \phi_i(v) = c) \\ 
    &\geq & \alpha \Pr(\phi^*(v_i) = \phi_i(v_i) | \phi_i(v) = c).
    \end{eqnarray*}
    Now, let $\Phi$ be the set of $L$-colorings $\phi_0$ of $H_i$ for which $\phi_0(v) = c$.
    Then, the event $\phi_i(v) = c$ is the disjoint union of the event set $\{\phi_i = \phi_0: \phi_0 \in \Phi\}$. Hence, Lemma \ref{lem:conditional} implies that 
    \begin{eqnarray*}
    \Pr(\phi_i(v) = c | \phi_i(v_i) = \phi^*(v_i)) 
    & \geq & \alpha \sum_{\phi_0 \in \Phi} \Pr(\phi^*(v_i) = \phi_i(v_i) | \phi_i = \phi_0) 
    \Pr(\phi_i = \phi_0| \phi_i(v) = c)     \\
    &=&  \alpha \sum_{\phi_0 \in \Phi} \Pr(\phi^*(v_i) = \phi_0(v_i)) \Pr(\phi_i = \phi_0| \phi_i(v) = c) \\
    &\geq & \alpha \beta \sum_{\phi_0 \in \Phi} \Pr(\phi_i = \phi_0| \phi_i(v) = c)  = \alpha \beta.
    \end{eqnarray*}
    Therefore,
    our distribution on $L$-colorings of $H$ satisfies (FIX), and thus
    $H$ is $(f,3,\alpha \beta)$-reductive, completing the proof.
\end{proof}

Next, we borrow a lemma from \cite{DNP} about flexible list colorings of paths.
\begin{lemma}[\cite{DNP}]
\label{lem:path-2}
    If $P$ is a path and $L$ is a $2$-assignment on $P$, then there exists a set $\Phi$ of exactly two $L$-colorings of $P$ such that for each $v \in V(P)$ and $c \in L(v)$, $\phi(v) = c$ for exactly one coloring $\phi \in \Phi$.
\end{lemma}
Lemma \ref{lem:path-2} has a useful corollary.
\begin{lemma}
\label{lem:path-is-flex}
    Let $P$ be a path, and let $f:V(P) \rightarrow \{2,3\}$ be a function. Then, $P$ is $(f,3,\frac{1}{3})$-reductive.
\end{lemma}
\begin{proof}
    Let $L$ be an $f$-assignment on $V(P)$. For each vertex $p \in V(P)$, we define a list $L'(p) \subseteq L(p)$ by taking a size $2$ subset of $L(p)$ uniformly at random.
    Then, by Lemma \ref{lem:path-2}, we  randomly choose an $L'$-coloring $\phi$ of $P$ so that for each vertex $p \in V(P)$ and each color $c' \in L'(p)$, 
    $\phi(p) = c'$ with probability $\frac{1}{2}$. Since each color $c \in L(p)$ appears in $L'(p)$ with probability at least $\frac{2}{3}$, it follows that $\phi(p) = c$ with probability at least $\frac{1}{3}$.
\end{proof}

Our final lemma proves that a $2$-connected graph $H$ of maximum degree $3$ is $(f,3,3^{-8})$-reductive under certain reasonable conditions on $H$ and $f$. 
The proof of the lemma follows the proof of \cite[Theorem 2.4]{BMS} very closely. 

\begin{lemma}
\label{lem:max3}
    Let $H$ be a $2$-connected graph of maximum degree $3$, and let $f:V(H) \rightarrow \mathbb N$ be a function. 
    Suppose that there exists a subset $X \subseteq V(H)$ of at most $1$ vertex so that
     the following properties are satisfied:
    \begin{itemize}
        \item $H$ is isomorphic to none of the following: $K_4^-, K_4, H_5, H_7$;
        \item $f(v) = 3$ for each $v \in V(H) \setminus X$;
        \item $\deg_H(x) = f(x) = 2$ for $x \in X$;
        %\item $H$ contains at least one vertex $b$ for which $\deg(b) = 2$ and $f(b) = 3$.
    \end{itemize}
    Then,  $H$ is $(f,3,3^{-8})$-reductive.
\end{lemma}
%We note that if $H$ is isomorphic to a diamond but all other conditions of the lemma hold, then the lemma's conclusion may be false. Indeed, if $H$ is a diamond and $X$ contains a single  vertex of degree $2$, then a single coloring preference at
%the vertex $v \in V(H) \setminus X$ satisfying $\deg_H(v) = 2$ cannot necessarily be satisfied. 
\begin{proof}
    Since $H$ is $2$-connected and has maximum degree $3$, every vertex of $H$ has degree $2$ or $3$.
    %Since $H$ contains a vertex $b$ satisfying $\deg_H(b) = 2$, $H$ is not isomorphic to $K_4$. Thus,
    %since $H$ is connected and has maximum degree $3$, $H$ has no $K_4$ subgraph.
    Let $L$ be an $f$-assignment on $H$.
    By Lemma \ref{lem:ERT}, $H$ is $L$-colorable. We aim to construct a probability distribution on $L$-colorings of $H$
    so that for each vertex $v \in V(H)$ and color $c \in L(v)$, $c$ is assigned to $v$ with probability at least $3^{-8}$.
    We describe a random procedure for producing an $L$-coloring $\phi$ of $H$.
    First, we assume that $V(H)$ has a predetermined linear order.
    Using this linear order, we 
    give $V(H)$ a first-fit (greedy) coloring $\psi$
    which satisfies the property that two vertices within distance $8$ of each other receive distinct colors.
    Since $H$ has maximum degree $3$, $\psi$ uses at most $3 \cdot (2^8 - 1) + 1=766$ colors.

    Next, we choose a color class $R$ of $\psi$ uniformly at random. We define the subgraph $H'  = H \setminus R$ and a function $h:V(H') \rightarrow \mathbb N$ so that $h(v) = f(v) - 1$ if $v$ has a neighbor in $R$ and $h(v) = f(v)$ otherwise.
    Since each pair of vertices in $R$ has a mutual distance of at least $9$, each vertex $v \in V(H')$ has at most one neighbor in $R$, and hence $h(v) \geq \deg_{H'}(v)$.
    %Furthermore, if $b \in V(H')$ is a vertex 
    %satisfying 
    %$\deg_H(b) = 2$ and
    %$f(b) = 3$, then $h(b) > \deg_{H'}(b)$.

    For each component $C$ of $H'$, we say that $C$ is a \emph{good component} if $C$ is $h$-choosable; otherwise, we say that $C$ is a \emph{bad component}. We make several claims about the structure of the bad components in $H'$.
    
    \begin{claim}
    \label{claim:two-blocks}
        Each bad component of $H'$ has at least two blocks.
    \end{claim}
    \begin{proof}
        Suppose that $H'$ has some bad component $C$ which is a single block.
        If $C$ is isomorphic to $K_2$, then as $C$ is a bad component, both vertices $v \in V(C)$ satisfy $\deg_H(v) = f(v) = 2$, a contradiction. Therefore, $C$ is $2$-connected, and by Lemma \ref{lem:ERT},
        $C$ is a cycle of an odd length $q \geq 3$ such that $h(v) = 2$ for each $v \in V(C)$.
        Since $V(H)$ has 
        at most one vertex $x$ satisfying $\deg_H(x) = f(x) = 2$,
        it follows that
        at least $q-1$ vertices of $C$ have a neighbor in $R$. 
        Note that if two distinct vertices of $R$ have a neighbor in $C$, then we can find distinct vertices $r,r' \in R$ at distance at most $3$ in $H$, a contradiction. Therefore, there exists a unique vertex $r \in R$ for which $r$ has a neighbor in $C$. 
        
        Now, we consider two cases.
        First, suppose that all $q$ vertices in $C$ are adjacent to $r$. If $q = 3$, 
        then $H[V(C) \cup \{r\}]$ is a $K_4$. As $H$ has maximum degree $3$, it follows that $H$ is isomorphic to $K_4$, a contradiction. If $ q \geq 5$, then $\deg_H(r) \geq 5$, a contradiction.

        Next, suppose that exactly $q-1$ vertices of $C$ are adjacent to $r$. 
        If $q \geq 5$, then $\deg_H(r) \geq 4$, a contradiction. If $q = 3$, then we use the fact that
        $\left (\bigcup_{v \in V(C)} N(v) \right ) \setminus V(C) = \{r\}$ as follows. If $\deg_H(r) = 3$, then $r$ is a cut vertex in $H$, contradicting the $2$-connectivity of $H$. If $\deg_H(r) = 2$, then $H$ is isomorphic to $K_4^-$, a contradiction. Therefore, $C$ has at least two blocks.
    \end{proof}

    Using Claim \ref{claim:two-blocks}, we can prove an even stronger claim.
    \begin{claim}
    \label{claim:three-blocks}
        Every bad component of $H'$ has at least three terminal blocks.
    \end{claim}
    \begin{proof}  
        Consider a bad component $C$ of $H'$.
        By Claim \ref{claim:two-blocks}, the block-cut tree of $C$ is not a single vertex, so $C$ has at least two terminal blocks. In order to obtain a contradiction, we assume that $C$ has exactly two terminal blocks.
        Since $C$ is a bad component, 
        Lemma \ref{lem:ERT} implies that each vertex $v \in V(C)$ satisfies $\deg_C(v) = h(v)$.

         Let $R_C \subseteq R$ be the set of vertices in $R$ with a neighbor in $C \setminus \{x\}$.
        We first argue that $|R_C| = 1$.
        As each terminal block of $C$ has a vertex with a neighbor in $R_C$,
        $|R_C| \geq 1$.
        To show that $|R_C| = 1$,        we define a graph $A_3$ whose vertex set consists of those vertices $v \in V(C)$ satisfying $\deg_C(v) \leq 2$.
        We let two vertices $u,v \in V(A_3)$ be adjacent in $A_3$ if and only if $\dist_C(u,v) \leq 3$. 
        %Note that since $C$ is a bad component, each vertex $v \in V(A_3)$
        %satisfies $h(v) = \deg_C(v)$. 
        
        We claim that $A_3$ is connected.
        To prove this claim, we first observe that as $C$ has exactly two terminal blocks, the block-cut tree of $C$ is a path. 
        We let $B_0$ be a terminal block of $C$, and we label the blocks of $C$ as $B_0, B_1, \dots, B_t$, where $B_t$ is the second terminal block of $C$, and blocks $B_i$ and $B_{i+1}$ are joined by a cut vertex.
        Since each block $B_i$ of $C$ is an odd cycle or a $K_2$, if a component $A$ of $A_3$ contains one vertex in $V(A_3) \cap V(B_i)$, then $A$ contains every vertex in $V(A_3) \cap V(B_i)$.

        %Now, suppose that $A_3$ is not connected.
        Since $B_0$ is a terminal block of $C$,
        Lemma \ref{lem:ERT} tells us that $B_0$ is isomorphic to $K_2$ or 
        an odd cycle, so
        $B_0$ contains a vertex $v_0$ satisfying $\deg_C(v_0) \leq  2$.
        We let $A$ be the component of $A_3$ containing
        $v_0$. We prove by induction on $m$ that $A$ contains every vertex in $V(A_3) \cap V(B_m)$ for each value $0 \leq m \leq t$, which implies that $A_3$ is connected. For our base case, when $m = 0$, there is nothing to prove. Now, suppose that $m \geq 1$.
        If each vertex $v \in V(B_m)$ satisfies $\deg_C(v) = 3$, then we are done.
        Otherwise, some vertex $v \in V(B_m)$ satisfies $\deg_C(v) \leq 2$.
        We choose $v$ to be 
        within distance $1$ of $B_{m-1}$ in $C$.
        In order to show that $A$ contains every vertex in $V(B_m) \cap V(A_3)$, it is enough to show that $v \in V(A)$.

        If $B_{m-1}$ is an odd cycle, then 
        as $H$ has maximum degree $3$,
        $B_m \cong K_2$. Hence,
        $B_{m-1}$ contains a vertex $u$ satisfying $\deg_C(u) \leq 2$ at distance at most $2$ from $v$. By the induction hypothesis, $u \in V(A)$, so $v \in V(A)$.
        If $B_{m-1}$ is a $K_2$ containing a vertex $u$ satisfying $\deg_C(u) \leq 2$, then $u$ and $v$ are at distance at most $2$. Since $u \in V(A)$ by the induction hypothesis, $v \in V(A)$. If $B_{m-1}$ is a $K_2$ whose vertices all have degree at least $3$ in $C$,
        then $m \geq 2$, and $B_{m-2}$ and $B_m$ are both odd cycles.
        Then, we find a vertex $u \in V(B_{m-2})$ satisfying $\deg_C(u) = 2$ for which $\dist_C(u,v) \leq 3$. As $u \in V(A)$ by the induction hypothesis,
        $v \in V(A)$. Therefore, $A$ contains every vertex in $V(B_m) \cap V(A_3)$ for $0 \leq m \leq t$, and hence $A_3$ is connected.

        %{\PB If $x$ has a neighbor in $R_C$, then every vertex of $A_3$ has a neighbor in $R_C$.}
        Now, to finish our proof that $|R_C| = 1$, we define a graph $A_6$ on $V(A_3)$, so that two vertices $u,v \in V(A_3)$ are adjacent in $A_6$ if and only if $\dist_C(u,v) \leq 6$. We observe that the square graph $A_3^2$ is a subgraph of $A_6$. Since $A_3$ is connected, it follows that $A_6$ is $2$-connected. 
        We then obtain a graph $A_6'$ 
        by deleting a vertex $x \in V(A_6)$
        if and only if $x \in X$. 
        Since at most one such vertex $x \in X$ exists,
        $A_6'$ is a connected graph. Furthermore, for each vertex $v \in V(A_6')$, $v$ has a neighbor in $R_C$, and each vertex in $R_C$ has a neighbor in $A_6'$.
        % {\RB is it possible for there to exist $r \in R_C$ such that the neighborhood of $r$ in $C$ is only $\{x\}$?}.
        % {\PB No. I am using the general fact that the square of a connected graph is $2$-connected and therefore has min deg at least $2$. Maybe this is worth citing.}
        If $|R_C| \geq 2$, then as $A_6'$ is connected, there exist vertices $r,r' \in R_C$ with adjacent neighbors in $A_6'$, implying that $\dist_H(r,r') \leq 8$, a contradiction.
        Therefore, $|R_C| =1$.

        Finally, as $|R_C| = 1$, $R_C$ has at most three neighbors in $C$. 
        Since $C$ contains at most one vertex $x$ satisfying $\deg_H(x) = 2$, it follows that $\sum_{v \in V(C)} (3-\deg_C(v)) \leq 4$. As each terminal block of $C$ is a $K_2$ or an odd cycle, each terminal block of $C$ contributes at least $2$ to this sum, so we conclude that $\sum_{v \in V(C)} (3-\deg_C(v)) =4$ and that for each non-terminal block $B$ of $C$, every vertex $v \in V(B)$ satisfies $\deg_C(v) = 3$. Using the argument that proves that $A_3$ is connected, each pair of adjacent non-terminal blocks of $C$ contains 
        a vertex $v$ for which $\deg_C(v) =2$, so we conclude that $C$ has at most one non-terminal block. Then, as $\sum_{v \in V(C)} (3-\deg_C(v)) = 4$, it follows that the two terminal blocks of $C$ are isomorphic either to $K_3$ and $K_2$ or to $K_3$ and $K_3$.
        In the first case, $C$ is isomorphic to $H_5 \setminus \{w\}$, and in the second case, $C$ is isomorphic to $H_7 \setminus \{v\}$. Then, as $H$ is $2$-connected and contains no vertex of degree $1$ and at most one vertex of degree $2$, it follows that $H$ is isomorphic to $H_5$ or $H_7$, a contradiction. 
        Therefore, $C$ has at least three terminal blocks.
    \end{proof}

    For a bad component $C$ of $H'$, we say that a terminal block $B$ of $C$ is \emph{nice} if
    $V(B) \cap X = \emptyset$.

    \begin{claim}
    \label{claim:1r}
        If $C$ is a bad component of $H'$ and $B$ is a nice terminal block of $C$, then $B$ has exactly one neighbor $r \in R$.
        Furthermore, 
        if $r$ has a neighbor in a nice terminal block $B'$ of a bad component of $H'$, then $B' = B$.
    \end{claim}
    \begin{proof}
        Let $C$ be a bad component of $H'$, and let $B$ be a nice terminal block of $H'$. Since $B$ is nice, $B$ is an odd cycle of length at least $3$, and all vertices $v \in V(B)$
        except one
        satisfy $\deg_C(v) = 2$.
        Therefore, the vertices of $V(B)$
        have a single neighbor $r \in R$, as otherwise, two distinct vertices in $R$ have mutual distance at most $3$, a contradiction.

        Now, suppose that $r$ has a neighbor $u$ in a nice terminal block $B'$ of some bad component $C'$ of $H'$, and suppose that $B' \neq B$.
        Since $\deg_H(r) \leq 3$, $r$ has exactly one neighbor 
        $v' \in V(B')$. However, $v'$ has a neighbor $u' \in V(B')$ satisfying $\deg_{C'}(u') = 2$, which in turn has
        a neighbor $r' \in R$ distinct from $r$. Then, $r$ and $r'$ have a mutual distance of at most $3$, a contradiction.
    \end{proof}

    Now, 
    we initialize a set $R' = R$.
    For each bad component $C$ of $H'$, we choose a nice terminal block $B_C$ of $C$
    uniformly at random. Then, 
    we let $r$ be the unique neighbor $r \in R$ of $B_C$, and we update $R' \leftarrow R' \setminus \{r\}$.
    After repeating this process for each bad component $C$,
    we define a new function $h':V(H \setminus R') \rightarrow \mathbb N$ so that $h'(v) = f(v) -  1$ 
    if $v$ has a neighbor in $R'$ and $h'(v) = f(v)$ otherwise.
    Each vertex $v \in V(H) \setminus R'$ satisfies
$h'(v) \geq \deg_{H \setminus R'}(v)$,
    and if $v \in V(H) \setminus R'$ satisfies $\deg_H(v) = 2$ and $f(v) = 3$, then 
    $h'(v) > \deg_{H \setminus R'}(v)$.
    We now say that a component $C$ of $H \setminus R'$ is \emph{good} if $C$ is $h'$-choosable.
    We argue that every component of $H \setminus R'$ is good.

    Let $C'$ be a component of $H \setminus R'$, and let $C$ be a component of $H'$ which is a subgraph of $C'$. If $C$ is a good component in $H'$, then Lemma \ref{lem:ERT} tells us that $C$ either contains a vertex $v$ satisfying 
    $f(v) > \deg_H(v)$, or $C$ contains an induced even cycle or theta subgraph. In both cases, $C'$ is a good component in $H \setminus R'$. 
    %{\RB To check my understanding, is this correct logic for the above two sentences: If $C$ is a good component in $H'$, then Lemma \ref{lem:ERT} tells us that there exists a vertex $v \in V(C)$ such that $h(v) > \deg_H'(v)$ or $C$ contains an induced even cycle or theta subgraph. Then $h'(v) > \deg_{H \setminus R'}(v)$ or $C'$ contains an induced even cycle or theta subgraph, so $C'$ is a good component in $H \setminus R'$.}
    %{\PB Yes, your understanding is correct.}
    If $C$ is a bad component of $H'$, then by Claim \ref{claim:1r}, $B_C$ has exactly one neighbor $r \in R$, and hence $B_C$ is a $K_3$ with vertices $u,v,w$ such that $\deg_C(w) = 3$ and $\deg_C(u) = \deg_C(v) = 2$.
    Then, $V(B_C) \cup \{r\}$ induces a diamond subgraph of $C'$, which implies by Lemma \ref{lem:ERT} that $C'$ is a good component of $H \setminus R'$.

    Finally,
    to produce our $L$-coloring of $H$,
    we first choose a color $c \in L(r)$ uniformly at random
    for each vertex $r \in R'$, and we assign $\phi(r) = c$. Then, since each uncolored component of $H \setminus R'$ is good, we complete an $L$-coloring on the remaining vertices of $H$.

    By Claims \ref{claim:three-blocks} and \ref{claim:1r},
    a vertex $r \in R$ belongs to $R'$ with probability at least $\frac{1}{2}$.
    Therefore,
    each vertex $v \in V(H)$ belongs to the set $R'$ with probability at least  $\frac{1}{2} \cdot \frac{1}{766} = \frac{1}{1532} > 3^{-7}$.
    %$\frac{1}{2} \cdot \frac{1}{385} = \frac{1}{770} > 3^{-7}$.
    Given that $v \in R'$, each color $c \in L(v)$ is assigned to $v$ with probability at least $\frac{1}{3}$. Therefore, $H$ is $(f,3,3^{-8})$-reductive.
\end{proof}

\section{Proof of Theorem \ref{thm:mad3}}
\label{sec:discharging}
The goal of this section is to prove that there exists a constant $\epsilon > 2^{-30}$ such that if $G$ is a graph of maximum average degree less than $3$, then $G$ is 
weighted
$\epsilon$-flexibly $3$-choosable, thereby proving Theorem \ref{thm:mad3}. The main tool for our proof is Lemma \ref{lem:main-k-red}.

We fix a graph $G$ satisfying $\mad(G) < 3$.
In order to prove that $G$ is weighted 
$\epsilon$-flexibly $3$-choosable, we may consider each component of $G$ separately, so we assume that $G$ is connected.
%Hence, whenever we refer to definitions of strong $k$-reducibility, we will assume that $k = 3$. 
For ease of notation, if $H$ is an induced subgraph of $G$, then we write $\ell_H(v) = \ell_{H,3}(v)
 = 3 - \deg_G(v) + \deg_H(v)$ for each vertex $v \in V(G)$.
In this way, if $L$ is a $3$-assignment on $G$
and $H$ is a subgraph of $G$, then for any $L$-coloring of $G \setminus H$, $\ell_H(v)$ gives a lower bound on the number of available colors in $L(v)$ for each vertex $v \in V(H)$.

We write $\alpha = 3^{-9}$ and $\epsilon = 
(\frac{2\alpha}{3})^{2} > 2^{-30}$.
By Lemma \ref{lem:main-k-red}, 
in order to prove that $G$ is $\epsilon$-flexibly $3$-choosable,
it suffices to show that every induced subgraph of $G$ contains an induced  $(3,\epsilon,\alpha)$-reducible subgraph.
In order to show that a given induced subgraph $H$ of $G$
is $(3,\epsilon,\alpha)$-reducible,
we typically check that 
for every $\ell_H$-assignment $L$ on $H$, there exists a probability distribution on $L$-colorings $\phi$ of $H$ such that for each $v \in  V(H)$ and $c \in L(v)$, the probability that $\phi(v) = c$ is at least $\alpha$,
and we also check that $\ell_H(v) \geq 2$ for each $v \in V(H)$.
As previously observed,
these two conditions
imply that $H$ is $(\ell_H,3,\alpha)$-reductive
and thus imply that $H$ is $(3,\epsilon,\alpha)$-reducible.
Alternatively, we may also show that $H$ is a terminal block of $G$ which is $(3,3,\frac{1}{3})$-reductive, which also implies that $H$ is $(3,\epsilon,\alpha)$-reducible.

Note that
every induced subgraph of $G$ has maximum average degree less than $3$. Hence, as $G$ is an arbitrary graph satisfying $\mad(G) < 3$, 
in order to prove that every induced subgraph of $G$ has an induced subgraph $H$ which is $(3,\epsilon,\alpha)$-reducible,
it suffices only to prove that $G$ contains 
an induced subgraph $H$ which is  $(3,\epsilon,\alpha)$-reducible. To prove this claim, we assume the contrary, which eventually leads us to a contradiction:
\begin{assumption}
\label{a}
No induced subgraph of $G$ is $(3,\epsilon,\alpha)$-reducible.
\end{assumption}

%In order to show that a given induced subgraph $H$ of $G$
%is $(3,\epsilon,\alpha)$-reducible,
%we will check that 
%for every $\ell_H$-assignment $L$ on $H$, there exists a probability distribution of $L$-colorings $\phi$ of $H$ such that for each $v \in  V(H)$ and $c \in L(v)$, the probability that $\phi(v) = c$ is at least $\alpha$,
%and we will check that $\ell_H(v) \geq 2$ for each $v \in V(H)$.
%As previously observed,
%these two conditions imply both the (FIX) and (FORB) statements from our reducible subgraph definition
%(Definition \ref{def:fixforb})
%when $k = 3$ and thus imply that $H$ is $(3,\epsilon,\alpha)$-reducible.

The main strategy for our proof is 
to 
use 
Assumption \ref{a}
in order to
establish a set
of
structural conditions and 
forbidden subgraphs in $G$.
Then, we use a discharging argument to show that $G$ has maximum average degree at least $3$, giving us our contradiction.

\subsection{Structural arguments}
In the following lemmas, we prove several structural conditions of $G$ which follow from
from Assumption \ref{a}.
\begin{lemma}
\label{lem:delta2}
    Each vertex of $G$ has degree at least $2$.
\end{lemma}
\begin{proof}
    We show that if a vertex $v \in V(G)$ has degree at most $1$, then $H = \{v\}$ is an induced subgraph of $G$ which is $(\ell_H,3,\frac{1}{3})$-reductive.
        Since $v$ has at most one neighbor in $G$, $ \ell_H(v) \geq 2$.
    Furthermore, if $L(v)$ is a list of $\ell_H(v)$ colors, then the uniform distribution 
    on $L(v)$ gives $H$ an $L$-coloring in which each $c \in L(v)$ appears at $v$ with probability at least $\frac{1}{3}$.
   Therefore, $H$ is 
   $(\ell_H,3,\frac{1}{3})$-reductive and hence
   $(3,\epsilon,\alpha)$-reducible, a contradiction.
\end{proof}
\begin{lemma}
\label{lem:terminal_diamond}
$G$ does not have a terminal block isomorphic to a diamond, $K_4$, $H_5$, or $H_7$.
\end{lemma}
\begin{proof}
As $\mad(G) < 3$, $G$ does not have a $K_4$ subgraph.
Suppose that $B$ is a terminal block of $G$ isomorphic to a diamond, $H_5$, or $H_7$. By Lemmas \ref{lem:diamond},
\ref{lem:H5}, and \ref{lem:H7}, $B$ is $(3, \epsilon, \alpha)$-reducible, a contradiction.
\end{proof}

\begin{lemma}
\label{lem:terminal-block}
    $G$ does not have a terminal block consisting entirely of vertices of degree $2$ or $3$.
\end{lemma}
\begin{proof}
    Suppose that $G$ has a terminal block $B$ consisting entirely of vertices of degree $2$ or $3$.
    By Lemma \ref{lem:terminal_diamond}, $B$ is not isomorphic to a diamond, $K_4$, $H_5$, or $H_7$.
    
    We first argue that each vertex $v \in V(B)$ satisfies $\deg_B(v) \geq 2$. 
    Indeed, if 
    some $v \in V(B)$ satisfies $\deg_B(v) \leq 1$, then as $B$ is a block, $B$ is a $K_2$ block.
    However, as $B$ is a terminal block of $G$, $G$ has a vertex of degree $1$, contradicting Lemma \ref{lem:delta2}. 
    Therefore, every vertex of $B$ has at least two neighbors in $B$.

    Now, we argue that $B$ is $(\ell_B,3,3^{-8})$-reductive.
    Since $B$ is a terminal block of $G$,
    $\ell_B(w) = 3$ for all vertices $w \in V(B)$ except possibly for a single cut vertex $x \in V(B)$. 
    If $x$ is a cut vertex of $B$, then $x$ has two neighbors in $B$ and a neighbor outside of $B$, 
    so $\deg_G(x) \geq 3$. By the assumption of the lemma, 
    $ \deg_G(x) = 3$. Thus, $x$ has exactly one neighbor in $G \setminus B$; therefore, $\ell_B(x) = 2$.
    Thus, by Lemma \ref{lem:max3}, $B$ is $(\ell_B,3,3^{-8})$-reductive and hence $(3,\epsilon,\alpha)$-reducible, a contradiction.
\end{proof}

We define a \emph{conductive path} in $G$ as a path
whose internal vertices are all of degree $3$ in $G$.
We say that two vertices $u,v \in V(G)$ are \emph{conductively connected} if there exists
a conductive path with endpoints $u$ and $v$.
A vertex is conductively connected with itself.
During our upcoming discharging argument,
we let charge flow along conductive paths between vertices.

Our next lemma resembles
a lemma of Dvo\v{r}\'ak, Masa\v{r}\'ik, Mus\'ilek, and Pangr\'ac \cite{DMMP6} which aims to prove that a planar graph of girth at least $6$ is flexibly $3$-choosable.
Namely, \cite[Lemma 6]{DMMP6} roughly
states that if two vertices of degree $2$ are joined by a 
bounded-length path $P$
whose internal vertices all have degree $3$, then the host graph has a subgraph which is $(3,\epsilon,\gamma)$-reducible
for some constant $\gamma$ depending on the length of $P$.
The proof of the following lemma shows that in fact $\gamma$ does not need to depend on the length of $P$, and hence such a path $P$ is forbidden as a subgraph of $G$.

\begin{lemma}
\label{lem:conductive_path}
    No two distinct vertices $u,v \in V(G)$ of degree $2$ are conductively connected.
\end{lemma}
\begin{proof}
   Suppose that $G$ contains two distinct
   conductively connected vertices $u,v$ of degree $2$.
   Let 
    $P$ be the shortest conductive path joining $u$ and $v$. Since $P$ is chosen to be shortest, $P$ is an induced subgraph of $G$. 

   Let $L$ be a $\ell_P$-assignment on $V(P)$. 
   Since $\deg_G(u) = \deg_G(v) = 2$, $\ell_P(u) = \ell_P(v) = 2$.
   Furthermore, each internal vertex $p \in V(P)$ satisfies $\deg_G(p) = 3$ and $\deg_P(p) = 2$, so $\ell_P(p) \geq 2$.
   %For each vertex $x \in V(P)$, we define $L'(x)$ to be a size $2$ subset of $L(x)$ chosen uniformly at random. Then, by Lemma \ref{lem:path-is-flex}, $P$ is strongly $\frac{1}{2}$-flexibly $L'$-choosable. For each vertex $x \in V(P)$ and color $c \in L(x)$, $c$ has a chance of at least $\frac{2}{3}$ of belonging to $L'(x)$, 
   Hence, by Lemma \ref{lem:path-is-flex}, $P$ is $(\ell_P,3,\frac{1}{3})$-reductive and hence $(3, \epsilon, \alpha)$-reducible, a contradiction.
\end{proof}

The remaining lemmas aim to
establish additional properties of $G$ which help us during our upcoming discharging argument.
In order to motivate these lemmas, we  sketch our discharging procedure. We assign each vertex $v \in V(G)$ a charge of $\deg(v) - 3$, so that the overall charge in $v$ is negative. 
%We will let each each vertex $x$ of degree $2$ in a pendant diamond $D$ take a charge of $1$ from the conductive vertex of $D$, so that all conductive and deadweight vertices have charge $0$. Then, 
Then, we aim to redistribute charge between vertices of $G$ so that the final charge of each vertex is nonnegative, giving us a contradiction. When redistributing charge, we let
vertices of degree at least $4$ give away charge, and we let vertices of degree $2$ receive charge. 
The
charge flows between vertices along conductive paths.

In order to let each vertex have a final nonnegative charge, our main challenge is to ensure that each vertex $v \in V(G)$ of degree $2$ in $G$ receives a total charge of $1$, while not letting any vertex $w \in V(G)$ of degree at least $4$ give away more than $\deg(w) - 3$ charge. In order to achieve this goal, we need to develop a detailed understanding of the conductive paths between vertices of degree $2$ and vertices of degree at least $4$.
The following lemmas help us develop this understanding.

%show that each insulated vertex $u$ of degree $2$ in $G'$ is conductively connected with at least one vertex of degree at least $4$ in $G'$.
%Furthermore when $u$ is conductively connected with exactly one vertex of degree of degree at least $4$ in $G'$,
%we aim to identify certain structural properties of $G'$.

\begin{lemma}
\label{lem:geq-1-4}
    If a vertex $u \in V(G)$ has degree $2$, then $u$ is conductively connected with a vertex $v \in V(G)$ satisfying $\deg_{G}(v) \geq 4$.
\end{lemma}
\begin{proof}
    Suppose that the lemma does not hold. Let $W$ be the 
    set of all vertices in $G$ with which $u$ is conductively connected.
    If $W$ contains a second vertex $u'$ of degree $2$, then $G$ contains a conductive path joining $u$ and $u'$, contradicting Lemma \ref{lem:conductive_path}. If $W$ contains a vertex $v$ satisfying $\deg_{G}(v) \geq 4$, then the lemma is proven. Otherwise, 
    $u$ is the only vertex in $W$ whose degree is not $3$.
    Since $G$ is a connected graph,
    it follows that $W = V(G)$.
    Then, $\ell_{G}(v) = 3$ for each vertex $v \in V(G)$, and hence Lemma \ref{lem:max3} and Lemma \ref{lem:terminal_diamond}
    tell us that
    $G$ is $(\ell_{G},3,3^{-8})$-reductive.
    Hence, $G$ is $(3, \epsilon, \alpha)$-reducible,
    a contradiction.
\end{proof}

Let $u \in V(G)$ be a
vertex of degree $2$, and
let $U$ be the set of vertices of degree at least $4$
with which $u$ is conductively connected.
By Lemma \ref{lem:geq-1-4}, $U$ contains at least one vertex.
We say that $u$ is \emph{expensive} if $|U| = 1$.
Otherwise, we say that $u$ is \emph{cheap}.
Note that if a vertex $u$ is called expensive or cheap, then $\deg(u) =  2$.
The reason for this terminology is that during the upcoming discharging argument, 
a cheap vertex $u$ takes only a charge of $\frac{1}{2}$ from each vertex of degree at least $4$ with which $u$ is conductively connected; however, an expensive vertex $u$ takes a charge of $1$ from the unique vertex of degree at least $4$ with which $u$ is conductively connected.

The next lemmas aim to establish some structural properties related to cheap and expensive vertices in $G$ which are critical to our discharging argument.

\begin{lemma}
\label{lem:two-neighbors}
    If $u \in V(G)$ is an expensive vertex
    which is conductively connected with a vertex $v \in V(G)$ of degree at least $4$,
    then $u$ and $v$ belong to a common terminal
    block $B$ of $G$ such that each vertex $w \in B \setminus \{u,v\}$
    satisfies $\deg_G(w) = 3$.
    Furthermore, either $G = B$, or $v$ is a cut vertex of $G$.
    %{\PB And also $u$ has two neighbors in $B$, and also 
    %each conductive vertex in $B$ is only cc with $u$ and $v$.
    %The property I really need is that $\lambda_B(w) = 3$}
    %Then $u$ and $v$ belong to a common block $B$ of $G$ such that for each vertex $w \in B \setminus \{u,v\}$, $w$ is conductive.
    %{\PB I should really have the conclusion that there are two internally vertex-disjoint conductive paths joining $u$ and $v$.}
\end{lemma}
\begin{proof}
    Let $B$ be the subgraph of $G$ induced by
    the
    set of vertices with which $u$ is conductively connected. 
    By Lemma \ref{lem:conductive_path},
    $u$ is the only degree $2$ vertex in $B$. 
    Since $u$ is expensive, $v$ is the only vertex of degree at least $4$ in $B$. Therefore, each vertex in $B \setminus \{u,v\}$ has degree $3$ in $G$. Hence, by definition of conductive connectivity, if $w\in B \setminus \{v\}$,
    then $B$ contains all edges joining $w$ to a neighbor $w' \in N_{G}(w)$.

    If $B$ is not $2$-connected, then $B$ has at least two terminal blocks, and hence at least one terminal block $B_0$ not containing $v$. We write $x$ for the cut-vertex of $B_0$.
        We argue that $B_0$ is a terminal block in $G$. Indeed, 
    if $B_0$ is not a terminal block in $G$, there exists a vertex $w\in V(B_0) \setminus \{x\}$ with a neighbor $w' \in V(G)$ for which $w' \not \in V(B_0)$ and hence $w' \not \in V(B)$. However, this contradicts the property of $B$ that we observed above.
    Hence, $B_0$ is a terminal block of $G$ 
    which consists entirely of vertices of degrees $2$ and $3$ in $G$, contradicting
    the Lemma \ref{lem:terminal-block}.
    Thus, we conclude that $B$ is $2$-connected, and the same argument implies that $B$
    is a terminal block of $G$ either with no cut vertex or with $v$ as its cut vertex. This proves the lemma.
\end{proof}

\begin{lemma}
    \label{lem:1exp}
    If $v \in V(G)$ is a vertex of degree $4$, then $v$ is conductively connected with at most one expensive vertex.
\end{lemma}
\begin{proof}
    Suppose that $v$ is conductively connected with two expensive vertices.
    Then, Lemma \ref{lem:two-neighbors}
    implies that $G$ consists of two terminal blocks $B_1,B_2$ joined at $v$,
    where each $B_i$ consists
    entirely of $v$, one vertex of degree $2$, and vertices of degree $3$. In particular, each vertex $w \in V(B_i)$ satisfies $\deg_{B_i}(w) \leq 3$.
    Since Lemma \ref{lem:terminal_diamond} tells us that neither block $B_i$ is a diamond, $K_4$, $H_5$, or $H_7$,
    Lemma \ref{lem:max3} implies that each block $B_i$ is $(3,3,3^{-8})$-reductive.%{\RB should this be $(\ell_G, 3, 3^{-8})$-reductive?}.
    %{\PB Isn't $\ell_G(v) = 3$ for all vertices?}
    
    We apply Lemma \ref{lem:cut}
    with $H = G$,
    $H^* = \{v\}$, $H_1 = B_1$, and $H_2 = B_2$.
    Since $\ell_{G}(w) = 3$ for each vertex $w \in V(G)$,
    $H^*$ is $(\ell_G,3,\frac{1}{3})$-reductive,
    and therefore $G$ is $(\ell_{G},3,3^{-9})$-reductive and hence $(3, \epsilon, \alpha)$-reducible, a contradiction. 
\end{proof}

\begin{lemma}
\label{lem:exp-cheap}
Let $v \in V(G)$ be a vertex of degree $4$. If $v$ is conductively connected with an expensive vertex, then $v$ is not conductively connected with a cheap vertex.
\end{lemma}
\begin{proof}
    Suppose that there exists an expensive vertex $u_e$ and a cheap vertex $u_c$ such that $v$ is conductively connected to both $u_e$ and $u_c$.
    By Lemma \ref{lem:two-neighbors}, $G$ contains a terminal block $B$ containing $u_e$ and $v$
    such that 
    each vertex $w \in V(B) \setminus \{v\}$
    satisfies $\deg_{G}(w) = \deg_B(w) \leq 3$.
    We let $P$ be the shortest conductive path joining $v$ and $u_c$.
    We let $H = B \cup P$,
    and since $B$ is a terminal block with $v$ as its cut vertex, and since $P$ is chosen to be shortest, $H$ 
    is an induced subgraph of $G$.

    Now, we observe that for each $p \in V(P)$, $\ell_H(p) \geq 2$.
    Additionally, each vertex $w \in V(B) \setminus \{v\}$ satisfies $\ell_H(w) = 3$. Furthermore, $\ell_H(v) = 2$ if and only if $\deg_B(v) = 2$; otherwise, $\ell_H(v) = 3$.
    Therefore, Lemma \ref{lem:max3}
    tells us that
    $B$ is $(\ell_H,3,3^{-8})$-reductive.
    Furthermore,  $2 \leq \ell_H(p) \leq 3$ holds for each vertex $p \in V(P)$, so by
    Lemma \ref{lem:path-is-flex},
    $P$ is 
    $(\ell_H,3,\frac{1}{3})$-reductive.
    Then, Lemma \ref{lem:cut}
    tells us that $H$ is 
    $(\ell_H,3,3^{-9})$-reductive and hence $(3,\epsilon,\alpha)$-reducible, a contradiction.
\end{proof}

\begin{lemma}
\label{lem:dminus2}
    If $v \in V(G)$ is a vertex of degree $d \geq 4$, then $v$ is conductively connected with at most $d-2$ cheap vertices.
\end{lemma}
\begin{proof}
    Suppose that $v$ is conductively connected with $d-1$ cheap vertices $u_1, \dots, u_{d-1}$.
    For each vertex $u_i$, let $P_i$ be the shortest conductive path in $G$ joining $v$ and $u_i$. We show that $H = P_1 \cup \dots \cup P_{d-1}$ is an induced subgraph of $G$ which is $(3, \epsilon, \alpha)$-reducible.

    First, we claim that $H$ is an induced subgraph. Indeed, suppose that there exist two adjacent vertices $p_i \in V(P_i) \setminus \{v\}$ and $p_j \in V(P_j) \setminus \{v\}$, for $i \neq j$. Then, $u_i$ and $u_j$ are conductively connected, contradicting Lemma \ref{lem:conductive_path}.

    Now, we observe that for each vertex $w \in V(H)$, $2 \leq \ell_H(w) \leq 3$. Therefore, 
    by Lemma \ref{lem:path-is-flex},
    each path $P_i$ is $(\ell_H,3,\frac{1}{3})$-reductive.
    Hence, by Lemma \ref{lem:cut}, $H$ is $(\ell_H,3,\frac{1}{9})$-reductive.
    Therefore, 
    $H$ is $(3,\epsilon,\alpha)$-reducible, a contradiction.
\end{proof}

\begin{lemma}
\label{lem:5}
    If $v \in V(G)$ satisfies $\deg(v) = 5$,
    then $v$ is conductively connected with neither of the following:
    \begin{enumerate}
        \item Three cheap vertices and one expensive vertex;
        \item One cheap vertex and two expensive vertices.
    \end{enumerate}
\end{lemma}
\begin{proof}
\begin{enumerate}
    \item  First, suppose that $v$ is conductively connected with three cheap vertices $u_1, u_2, u_3$ and an expensive vertex $u^*$.
For each cheap vertex $u_i$, let $P_i$ be the shortest conductive path joining $v$ and $u_i$.
By Lemma \ref{lem:two-neighbors}, $G$ has a terminal block $B$ containing $u^*$ and $v$
such that each vertex $b \in V(B) \setminus \{v\}$ has degree at most $3$.
We claim that $H = P_1 \cup P_2 \cup P_3 \cup B$ is an induced subgraph of $G$ which is $(3,\epsilon,\alpha)$-reducible.

First, we argue that $H$ is an induced subgraph of $G$.
By the same argument used in Lemma \ref{lem:dminus2}, 
$P_1 \cup P_2 \cup P_3$ is an induced subgraph of $G$, as otherwise, two cheap vertices are conductively connected, contradicting Lemma \ref{lem:conductive_path}.
Since $B$ is a terminal block of $G$ with a cut-vertex $v$, $H$ is therefore an induced subgraph of $G$. 

Next, we observe that 
since $B$ is a block, $v$ has two neighbors in $B$. Hence, $\ell_H(v) = 3$, and so it follows from Lemma \ref{lem:path-is-flex} that
each path $P_i$ is 
$(\ell_H,3,\frac{1}{3})$-reductive.
 Furthermore,
$\ell_H(b) = 3$ for all $b \in V(B)$, so
since $B$ is not a diamond, $K_4$, $H_5$, or $H_7$ by Lemma \ref{lem:terminal_diamond},
it follows from Lemma \ref{lem:max3} that $B$ is $(\ell_H,3,3^{-8})$-reductive.
Hence, by Lemma \ref{lem:cut},
$H$ is 
$(\ell_H,3,3^{-9})$-reductive and hence $(3,\epsilon, \alpha)$-reducible, a contradiction.

\item  Next, suppose that $v$ is conductively connected with one cheap vertex $u$ and two expensive vertices $u^*_1$ and $u^*_2$. 
We let $P$ be the shortest conductive path joining $v$ and $u$, and for each $u^*_i$, we let $B_i$ be the block of $G$ containing $u_i^*$ and $v$. Then, we may follow a similar argument as in the previous case to show that $H = P \cup B_1 \cup B_2$ is an induced subgraph of $G$ which is 
$(\ell_H,3,3^{-9})$-reductive
and hence $(3, \epsilon, \alpha)$-reducible, a contradiction.
\end{enumerate}
\end{proof}
\begin{lemma}
\label{lem:st}
    If $v$ is a vertex satisfying $\deg(v) \geq 6$ that is conductively connected with $s$ expensive vertex and $t$ cheap vertices, then 
    $2s + t \leq \deg(v).$
\end{lemma}
\begin{proof}
    For each cheap vertex $u$ with which $v$ is conductively connected, we let $P_u$ be a conductive path in $G$ joining $u$ and $v$. 
    By Lemma \ref{lem:two-neighbors},
    if $u^*$ is an expensive vertex with which $v$ is conductively connected, then 
    $G$ has a terminal block $B$ containing both $u^*$ and $v$ for which every vertex in $B \setminus \{u^*,v\}$ is conductive.
    Using Menger's theorem,
    we define
    two edge-disjoint conductive paths $P_{u^*}$ and $P_{u^*}'$ 
    joining $u^*$ and $v$. 

    Now, if $2s + t > \deg(v)$, then by the pigeonhole principle, 
    there exist two paths $P_1$ and $P_2$ defined in the previous step and a single neighbor $w \in N(v)$ such that $vw \in E(P_1) \cap E(P_2)$.
    We let $u_1$ and $u_2$ be the endpoints of $P_1$ and $P_2$ apart from $v$, respectively.
    It cannot hold that $u_1 = u_2$,
    since this would imply that $u_1 = u_2$ is expensive,
        and $P_{u_1}$ and $P'_{u_1}$ are edge-disjoint by construction.
    Hence, $u_1$ and $u_2$ are distinct  vertices of degree $2$ in $G$ which are conductively connected, contradicting Lemma \ref{lem:conductive_path}.
\end{proof}

\subsection{Discharging}
Now, we use discharging to finish our proof of Theorem \ref{thm:mad3}.
Let each vertex $v \in V(G)$ receive a charge of $\deg(v) - 3$.
Since the average degree of $G$ is less than $3$, the total charge distributed throughout $G$ is negative.
We apply the following steps. 
\begin{enumerate}
    \item If $u$ is a cheap vertex, $u$ takes a charge of $\frac{1}{2}$ from each vertex of degree at least $4$ with which $u$ is conductively connected.
    \item If $u$ is an expensive vertex, then $u$ takes a charge of $1$ from the sole vertex of degree at least $4$ with which $u$ is conductively connected.
\end{enumerate}
Since each step conserves the total charge in $G$, after applying these steps, the total charge in $G$ is still negative. We show that after applying these steps,
each vertex in $G$ has a nonnegative charge, which gives us a contradiction.

Consider a vertex $v \in V(G)$. 
If $\deg(v) = 2$, then $v$ is either cheap or expensive. If $v$ is cheap, then $v$ by definition is conductively connected with at least two vertices of degree at least $4$. Hence, $v$ gains a charge of at least $2 \cdot \frac{1}{2} = 1$ during discharging, and the final charge of $v$ is at least $2 - 3 + 1 = 0$. If $v$ is expensive, then $v$ by definition is conductively connected with exactly one vertex of degree at least $4$. Hence, $v$ gains a charge of exactly $1$ during discharging, and the final charge of $v$ is $2 - 3 + 1 = 0$.

If $\deg(v) = 3$, 
then $v$ 
does not gain or lose any charge during discharging, so the final charge of $v$ is $3 - 3 = 0$.

If $\deg(v) = 4$,
then
$v$ gives away a charge of $1$ to each expensive vertex with which $v$ is conductively connected
and a charge of $\frac{1}{2}$ to each cheap vertex with which $v$ is conductively connected.
If $v$ is conductively connected with an expensive vertex $u$, then Lemmas \ref{lem:1exp} and \ref{lem:exp-cheap} imply that $u$ is the only vertex of degree $2$ with which $v$ is conductively connected. Therefore, $v$ gives away at most $1$ charge.
If $v$ is not conductively connected with an expensive vertex, then Lemma \ref{lem:dminus2} tells us $v$ is conductively connected with at most two cheap vertices, and hence $v$ gives away at most $1$ charge. In both cases, the final charge of $v$ is at least $4 - 3 - 1 = 0$.

If $\deg(v) = 5$,
then $v$ gives away a charge of $\frac{1}{2}$ to each cheap vertex conductively connected with $v$ and a charge of $1$ to each expensive vertex conductively connected with $v$. 
Lemmas \ref{lem:dminus2} and \ref{lem:5} tells us that 
if $v$ is conductively connected with $s$ expensive vertices and $t$ cheap vertices, then $s + \frac{1}{2}t \leq 2$. Hence, $v$ gives away at most $2$ charge, and the final charge of $v$ is at least $5 - 3 - 2 = 0$.

If $\deg(v) \geq 6$, then
$v$ gives away a charge of $\frac{1}{2}$ to each cheap vertex conductively connected with $v$ and a charge of $1$ to each expensive vertex conductively connected with $v$. 
Lemma \ref{lem:st} tells us that if $v$ is conductively connected with $s$ expensive vertices and $t$ cheap vertices, then $s + \frac{1}{2}t \leq \frac{1}{2}\deg(v)$. Therefore, the final charge of $v$ is
at least $\deg(v) - 3 - \frac{1}{2}\deg(v) \geq 0$.

Before discharging, the total charge distributed throughout $G$ is negative, but after discharging, the total charge distributed throughout $G$ is nonnegative. Since each step of discharging conserves total charge, this gives us a contradiction.
Therefore, we conclude that Assumption \ref{a} is incorrect and that $G$ contains an induced $(3,\epsilon,\alpha)$-reducible subgraph.
Furthermore, since $G$ is an arbitrary graph satisfying $\mad(G) < 3$, the same argument implies that every induced subgraph of $G$ contains an induced $(3,\epsilon,\alpha)$-reducible subgraph.
Thus, 
Lemma \ref{lem:main-k-red}
shows that $G$ is weighted $\epsilon$-flexibly $3$-choosable.
This completes the proof of Theorem \ref{thm:mad3}.

\section{Conclusion}
The condition that $\mad(G) < 3$ in
Theorem \ref{thm:mad3} is best possible, as $K_4$ has maximum average degree exactly $3$ and is not $3$-choosable. However, if we restrict our attention to graphs with no $K_4$ subgraph, then the question of whether the condition $\mad(G) < 3$ can be relaxed is open, and the correct answer is not clear. Recently,
the second author, along with Choi, Kostochka, and Xu \cite{k=4}, showed that every graph $G$ with maximum average degree at most $\frac{16}{5} = 3+\frac 15$ and no $4$-Ore subgraph on at most $10$ vertices (see \cite{Ore} for a definition) is $3$-choosable. In particular, every $K_4$-free graph $G$ with maximum average degree less than $\frac{22}{7} = 3+\frac{1}{7}$ is $3$-choosable. With this result in mind, we pose the following question.
\begin{question}
    What is the maximum value $d$ 
    for which there exists $\epsilon > 0$ 
    such that every $K_4$-free graph with maximum average degree less than $d$ is $\epsilon$-flexibly $3$-choosable?
\end{question}

\bibliographystyle{plain}
\bibliography{bib}

\end{document}